\documentclass[11pt,oneside,english,reqno]{amsart}
\usepackage[T1]{fontenc}
\usepackage[latin9]{inputenc}
\usepackage[a4paper]{geometry}
\usepackage{mathrsfs}
\usepackage{amstext}
\usepackage{amsthm}
\usepackage{amssymb}
\usepackage{color}
\usepackage{hyperref}
\usepackage{dsfont}
\usepackage{enumerate}
\usepackage{verbatim} 
\usepackage{stmaryrd}
\usepackage{enumitem}

\makeatletter
\numberwithin{equation}{section}
\numberwithin{figure}{section}
\theoremstyle{plain}
\newtheorem{thm}{\protect\theoremname}
\theoremstyle{definition}
\newtheorem{defn}[thm]{\protect\definitionname}
\theoremstyle{remark}
\newtheorem{rem}[thm]{\protect\remarkname}
\theoremstyle{plain}
\newtheorem{lem}[thm]{\protect\lemmaname}
\theoremstyle{plain}
\newtheorem{prop}[thm]{\protect\propositionname}
\theoremstyle{plain}
\newtheorem{cor}[thm]{\protect\corollaryname}
\theoremstyle{plain}

\makeatother

\usepackage{babel}
\providecommand{\corollaryname}{Corollary}
\providecommand{\definitionname}{Definition}
\providecommand{\lemmaname}{Lemma}
\providecommand{\propositionname}{Proposition}
\providecommand{\remarkname}{Remark}
\providecommand{\theoremname}{Theorem}
\providecommand{\examplename}{Example}

\newcommand{\blue}[1]{{\color{blue} #1}}

\newcommand{\cB}{\mathcal{B}}

\newcommand{\cD}{\mathcal{D}}

\newcommand{\cF}{\mathcal{F}}

\newcommand{\cP}{\mathcal{P}}

\newcommand{\cS}{\mathcal{S}}

\newcommand{\EE}{\mathbb{E}}

\newcommand{\NN}{\mathbb{N}}

\newcommand{\PP}{\mathbb{P}}

\newcommand{\RR}{\mathbb{R}}

\newcommand{\mathd}{\mathrm{d}}

\newcommand{\vertiii}[1]{{\left\vert\kern-0.25ex\left\vert\kern-0.25ex\left\vert #1 
    \right\vert\kern-0.25ex\right\vert\kern-0.25ex\right\vert}}

\newcommand{\eps}{\varepsilon}
\newcommand{\dd}{\mathop{}\!\mathrm{d}}

 \hypersetup{
     colorlinks   = true,
     citecolor    = blue,
     linkcolor    = blue
}

\title[Regularization of multiplicative SDEs  through additive noise]{Regularization of multiplicative SDEs\\ through additive noise}
\author{ Lucio Galeati \and Fabian A. Harang}

\thanks{ 
\\
\emph{Acknowledgments}: F. Harang gratefully acknowledges financial support from the STORM project 274410, funded by the Research Council of Norway.}

\address{Lucio Galeati: {\rm email: lucio.galeati@iam.uni-bonn.de},
address: Institute of Applied Mathematics, University of Bonn, 53115 Endenicher Allee 60, Bonn, Germany.
}

\address{Fabian A. Harang: {\rm email: fabianah@math.uio.no},
address:  Department of Mathematics, University of Oslo, P.O. box 1053, Blindern, 0316, OSLO, Norway}

\begin{document}

\maketitle
\begin{abstract}
    We investigate the regularizing effect of certain additive continuous perturbations on SDEs with multiplicative fractional Brownian motion (fBm). Traditionally, a Lipschitz requirement on the drift and diffusion coefficients is imposed to ensure existence and uniqueness of the SDE. We show that suitable perturbations restore existence, uniqueness and regularity of the flow for the resulting equation, even when both the drift and the diffusion coefficients are distributional, thus extending the program of regularization by noise to the case of multiplicative SDEs. 
    Our method relies on a combination of the non-linear Young formalism developed by Catellier and Gubinelli \cite{Catellier2016}, and stochastic averaging estimates recently obtained by Hairer and Li \cite{hairer2019averaging}. 
\end{abstract}
{
\hypersetup{linkcolor=black}
 \tableofcontents 
}

\section{Introduction}
In this paper we deal with multidimensional stochastic differential equations of the form
\begin{equation}\label{intro general SDE}
    \dd x_t = b_1(t,x_t)\dd t+b_2(t,x_t)\dd \beta_t+\dd w_t,\qquad x_0\in \RR^d,
\end{equation}
where $\beta$ is a fractional Brownian motion with Hurst parameter $H>1/2$ and $w$ is a deterministic continuous path. Specifically, we are interested in understanding how the additive perturbation affects the SDE, by identifying \textit{analytic} conditions on $w$ which ensure wellposedness for \eqref{intro general SDE} even when  it fails  for $w\equiv 0$, in the style of \textit{regularisation by noise} phenomena.

Let us first provide a short account of the main known results for \eqref{intro general SDE} with $w\equiv 0$.
Since $H>1/2$, the SDE is pathwise meaningful either in the sense of Young integrals or fractional calculus; for $b_1$ and $b_2$ sufficiently smooth, existence of a unique solution is classical, see e.g. \cite{rascanu2002differential,Friz2014}, as well as \cite[Appendix D]{biagini2008stochastic} for a general survey.
Sharp  conditions for wellposedness, in the form of Osgood-type regularity for $b_1$ and $b_2$, are given in \cite{YongPeiJiangLun2017}, generalizing to the case $H>1/2$ the results from \cite{yamada1981successive,taniguchi1992successive} for $H=1/2$; this includes the case of $b_1$ and $b_2$ Lipschitz.
If $d=1$ and $b_2\equiv 1$, the authors in \cite{NUALART2002103} establish pathwise uniqueness for $b_1$ satisfying suitable H\"older regularity\blue{. T}his result can be extended to a broader class of non-degenerate diffusion coefficients $b_2$ by means of a Doss-Sussman transformation, in the style of \cite{athreya2017smoothness}. Recently, \cite{hinz2020variability} investigated the case $b_1\equiv 0$ and $b_2$ non-degenerate of bounded variation; however, the conditions included therein for wellposedness are fairly specific and require verification for each choice of $b_2$.\\

None of the results mentioned above includes the case of general H\"older continuous diffusion $b_2$ and smooth drift $b_1$. This is not due to technical limitations of the proofs; in fact, uniqueness does in general not  hold.
To see this, let $d=1$ and consider $y$ solution to the ODE $\dot y_t = f(y_t)$ with $y_0=0$, and define the process $x_t:=y(\beta_t)$. Under the assumption that $f$ is $\alpha$-H\"older with $H(1+\alpha)>1$, Young chain rule shows that $x$ satisfies the SDE
\[
\mathd x_t 
= f(x_t) \mathd\beta_t, \quad x_0=0.
\]

As a consequence, to any solution of the ODE we can associate a solution of the SDE; if uniqueness fails for the first, it will also fail for latter. For instance we can take
\[
f(z)=\frac{1}{1-\alpha}\,|z|^\alpha,\quad y^1_t= 0,\quad y^2_t = t^{\frac{1}{1-\alpha}},
\]
which implies that $x^1_t=0$ and $x^2_t=(\beta_t)^{1/(1-\alpha)}$ are two different solutions starting from $0$ to the same SDE; the above procedure actually allows to construct infinitely many of them.\\

Therefore the wellposedness theory for SDEs driven by fBm with $H>1/2$ can not be better than the one for classical ODEs.
At the same time, since existence of solutions is granted by compactness arguments under mild regularity assumptions on $b_1$ and $b_2$, it is reasonable to ask whether, among the many mathematical solutions, some are more meaningful than others. If the SDE models a physically observed phenomenon, then its solutions intuitively should be stable under very small perturbations. In this sense, establishing uniqueness for \eqref{intro general SDE} with very small, nontrivial $w$, can be seen as the first step in this context of the more general program on vanishing noise selection of solutions outlined in~\cite{Flandoli2011}.\\

Investigations on wellposedness of the SDE \eqref{intro general SDE} with $w$ sampled as a stochastic process date back to the pioneering work of Zvonkin \cite{Zvonkin1974} and the literature on the topic has grown extensively, see e.g. \cite{Veretennikov1981,Krylov2005,flandoli2010well,mohammed2015,beck2019stochastic} and the review \cite{Flandoli2011}.
However, to the best of our knowledge, only the case $b_2 \equiv 0$ has been treated so far; the presence of a diffusion term, combined with the fact that in the regime $H>1/2$ many classical probabilistic  tools (martingale problems, Markov processes and generators) are not available, creates new difficulties and different sets of idea must be introduced.\\

Our approach to the problem
follows the ideas introduced in \cite{Catellier2016}, where analytic conditions on $w$ which imply well-posedness for \eqref{intro general SDE} with $b_2\equiv 0$ and possibly distributional drift $b_1$ are identified. In recent years, this analytic approach to regularization by noise phenomena has been considerably expanded, see~\cite{galeati2020noiseless,harang2020cinfinity,harang2020regularity}.

From now on, in order not to hinder the main contributions of this work with technical details, we will focus for simplicity on the addtively perturbed SDE (in integral form)
\begin{equation}\label{eq:short SDE}
    x_t =x_0+\int_0^t b(x_s)\dd \beta_s + w_t
\end{equation}
namely with $b_1\equiv 0$ and $b_2$ not depending on time, but being possibly distributional. Indeed \eqref{eq:short SDE} presents the same main difficulties and, once they are properly understood, generalising the results to \eqref{intro general SDE} is almost straightforward, as will be shown in Section \ref{sec: further extensions}.\\

Our main strategy is based on readapting the non-linear Young formalism introduced in \cite{Catellier2016} in this setting. Given a solution $x$ to \eqref{eq:short SDE}, $\theta:=x-w$ formally solves
\begin{equation}\label{eq:formal young}
    \theta_t=\theta_0+\int_0^t b(\theta_s+w_s)\dd \beta_s.
\end{equation}
If both $b$ and $w$ are sufficiently regular, then equation \eqref{eq:formal young} can be reinterpreted as a nonlinear Young differential equation (nonlinear YDE for short) of the form
\begin{equation}\label{eq:formal young 2}
\theta_t =\theta_0 + \int_0^t \Gamma^w b(\dd s,\theta_s),
\end{equation}
where we denote by $\Gamma^w b$ the \textit{multiplicative averaged field}, formally defined as
\begin{equation}\label{Gamma intro}
\Gamma^w b(t,y)=\int_0^t b(y+w_r)\dd \beta_r,\qquad t\in [0,T], \, y\in \RR^d.
\end{equation}
It plays in this context the same role as the \textit{classical averaged field} $T^w b$ from \cite{Catellier2016}, given by
\begin{equation*}
T^w b(t,y)=\int_0^t b(y+w_r)\dd r,\qquad t\in [0,T], \, y\in \RR^d. 
\end{equation*}
We can then \textit{define} $x$ to be a solution to \eqref{eq:short SDE} by imposing the ansatz $x=w+\theta$, with $\theta$ solution to \eqref{eq:formal young 2}; in this way we can give meaning to \eqref{eq:short SDE} for less regular choices of $b$ and $w$, assuming we are able to prove the required regularity for $\Gamma^w b$. Existence and uniqueness of $x$ then reduces to that of $\theta$, which in turn follows from the abstract theory of non-linear YDEs (see Section \ref{sec:Non-linear Young integration and equations} for a recap) applied to the random field $\Gamma^w b$.\\

There are however some major problems in achieving the program outlined above, compared to the case of perturbed ODEs treated in \cite{Catellier2016}. Indeed, the classical averaged field $T^w b$ is by now a well understood object, which is always analytically well defined as a distribution. Moreover, many stochastic estimates are available for $T^w b$ when $w$ is sampled as suitable stochastic processes, see Section \ref{sec2.1} for an overview.
In contrast, in order to define the integral appearing in \eqref{Gamma intro} as a Young integral, we need at least to require $w$ to be $\delta$-H\"older continuous with $H+\delta>1$; without this assumption, it is unclear how to interpret neither \eqref{eq:short SDE} nor \eqref{Gamma intro}, even when $b$ is a smooth function.
At the same time, it is now clear from \cite{Catellier2016,galeati2020noiseless,harang2020cinfinity} that a strong regularisation effect is expected to hold for especially rough $w$, i.e. for very small values of $\delta$, thus making the requirement $H+\delta>1$ too restrictive.

In order to overcome this difficulty, we must invoke recently developed stochastic estimates by Hairer and Li \cite{hairer2019averaging}, regarding Wiener integrals of the form
\[
\int_0^t f_s \dd \beta_s
\]
with $\beta$ fBm with $H>1/2$ and $f:[0,T]\to\mathbb{R}$ possibly distributional.
Remarkably, this not only allows to define $\Gamma^w b$ as a random field, but also relates its space-time H\"older regularity to that of $T^w b$, with no restrictions on the value $\delta\in (0,1)$. With this tool at hand, we can then apply the already existing results for $T^w b$ in order to define $\Gamma^w b$ and solve the associated equation \eqref{eq:formal young 2}.

Our approach presents several nice features: it identifies sufficient analytic conditions for $w$ to regularise the SDE, in the form of regularity requirements for $T^w b$; it provides a pathwise solution concept for \eqref{eq:short SDE} in terms of equation \eqref{eq:formal young 2}, which should be regarded as a random nonlinear YDE rather than an SDE; no adaptedness requirements are needed to guarantee uniqueness; finally, the existence of an associated Lipschitz flow is a direct consequence of the nonlinear YDE theory.

\subsection{Main results}\label{sec: main results}

In all the next statements, whenever referring to a fractional Brownian motion $\beta$ of parameter $H$, we will consider it to be the canonical process on $(\Omega,\mathcal{F},\mu^H)$, where $\Omega=C([0,T];\RR^m)$, $\mu^H$ is the fBm law on $\Omega$ and $\mathcal{F}$ is the completion of the $\mathcal{B}(C([0,T];\RR^m))$ w.r.t. $\mu^H$; the process $\beta=\{\beta_t\}_{t\in[0,T]}$ is given by $\beta_t (\omega)=\omega(t)$. However, as will be discussed, the concept of path-by-path wellposedness only depends on the law $\mu^H$, therefore the results automatically carry over to any other probability space $(\Omega,\mathcal{F},\PP)$ on which an fBm of parameter $H>1/2$ is defined. We will frequently refer to the averaged fields $T^wb$ and $\Gamma^wb$, formally given above and rigorously defined in Sections \ref{sec:prelim on nonlinear integration and avg fields} and \ref{sec: avg fields w multiplicative noise} respectively. \\

The following statement summarizes our main findings.

\begin{thm}\label{main thm1}
Let $H\in (1/2,1)$, $b\in \cD(\RR^d)$ and $w$ a deterministic path such that
\begin{equation}\label{main thm1 condition}
   T^w b\in C^\gamma_t C^2_x \text{ for some } \gamma\in \left( \frac{3}{2}-H,1\right); 
\end{equation}
then path-by-path wellposedness holds for the SDE
\begin{equation*}
    \mathd x_t = b(x_t)\mathd \beta_t +\mathd w_t.
\end{equation*}
In particular, for any $x_0\in\mathbb{R}^d$, any two pathwise solutions defined on $(\Omega,\mathcal{F},\mathbb{P})$ starting from $x_0$  are indistinguishable. Moreover, solutions are adapted to the filtration generated by $\beta$ and they form a random $C^1_{x,loc}$ flow; specifically, the unique solution starting at $x_0$ is given by
\begin{equation}\label{eq: form solution}
    x_t(\omega) = w_t + \mathcal{I}(\Gamma^w b(\omega))(t,x_0-w_0)
\end{equation}
where $\mathcal{I}(\Gamma^w b)$ is another random $C^1_{x,loc}$ flow.
\end{thm}

For the definitions of pathwise solution and path-by-path wellposedness, we refer to Section~\ref{sec: concepts solution}. Let us mention that pathwise solutions need not to be adapted, which is instead a consequence of Theorem~\ref{main thm1}; this is a non trivial fact, as there are SDEs for which path-by-path uniqueness holds but there exist no adapted solutions, see~\cite{shaposhnikov2020pathwise}.

A rigorous construction of the random field $\omega\mapsto \Gamma^w(\omega)$, together with its space-time regularity, is presented in Section~\ref{sec: avg fields w multiplicative noise}. The notation $\mathcal{I}(\Gamma^w b(\omega))$ is not by chance: as shown in Corollary~\ref{cor: continuous dependence flow map}, it's possible to define a map continuous $\mathcal{I}(\cdot)$ which maps drifts of prescribed regularity into flows. Therefore equation~\eqref{eq: form solution} implies that the solution map admits the following decomposition:
\[
\omega\mapsto \Gamma^w b(\omega) \mapsto \mathcal{I}(\Gamma^w b(\omega)) \mapsto x(\omega)
\]
where the first map is measurable, but the other ones are continuous; this is in a nice analogy with the classical decomposition of the It\^o-Lyons map from rough path theory.\\

A justification of our interpretation of the SDE in terms of a nonlinear YDE related to $\Gamma^w b$ comes from the next statement.

\begin{prop}\label{prop: justification solutions}
Let $H\in (1/2)$, $b$, $w$, $\beta$ as above. Then:
\begin{itemize}
    \item[i.] If $b$ and $w$ are regular, then any pathwise solution to the SDE
    \[
    x_t(\omega) = x_0 + \int_0^t b(x_s(\omega)) \mathd \beta_s(\omega) + w_t,
    \]
    where the integral is interpreted in the Young sense, is also a pathwise solution in the sense of Definition~\ref{defn: solution}.
    \item[ii.] If condition~\eqref{main thm1 condition} holds, then it's possible to find sequences $(b^n,w^n)$ of regular coefficients such that $(b^n,w^n)\to (b,w)$ and the associated pathwise solutions $x^n$ converge in probability to the unique pathwise solution $x$ given by Theorem~\ref{main thm1}.
    \item[iii.] More generally, if condition~\eqref{main thm1 condition} holds, for any sequence of regular coefficients $(b^n,w^n)\to (b,w)$ such that
    \[
    T^{w^n} b^n \text{ is Cauchy in } C^\gamma_t C^2_x \text{ for some } \gamma\in \left(\frac{3}{2}-H,1 \right)
    \]
    the associated pathwise solutions $x^n$ converge in probability to $x$.
\end{itemize}
\end{prop}

We have left some of the details of Proposition~\ref{prop: justification solutions} (the exact regularity, the notions of convergence, etc.) vague on purpose, as it should be regarded as some kind of meta theorem or general principle; more details will be given in the proof in Section~\ref{sec: proofs main results}.\\
Let us stress that condition $(b^n,w^n)\to (b,w)$ alone is \textit{not} enough to deduce $x^n\to x$! Indeed, if we mollify the path $w$ first, then its irregularity and its regularising effect on equation (measured by the regularity of $T^w b$) are completely lost; in order to build approximations schemes, one needs to first approximate $b$ by a more regular version $b^n$ and only then approximate $T^w b^n$ by $T^{w^n} b^n$, so that at each step the regularity of the averaged field is preserved.\\

Direct-to-check conditions on the regularity of $T^w b$, as well as higher regularity for the flow, are given by the next statement. 

\begin{thm}\label{main thm2}
Let $b\in C^\alpha_x$, $\alpha\in\mathbb{R}$, $w$ be such that $T^w b\in C^{1/2}_t C^{\alpha+\nu}_x$ for $\nu>0$ satisfying
\begin{equation}\label{main thm2 condition}
    \alpha + \nu (2H-1)>2.
\end{equation}
Then the hypothesis of Theorem~\ref{main thm1} are met. If in addition $T^w b\in C^{1/2}_t C^{\alpha+\nu}_x$ with
\begin{equation}\label{main thm2 condition2}
    \alpha + \nu (2H-1)>n+1,
\end{equation}
then the random flow associated to the SDE is $C^n_{x,loc}$.
\end{thm}

If both the diffusion coefficient $b$ and the perturbation $w$ are sufficiently regular to give meaning to the SDE as a classical Young differential equation, but not to establish its uniqueness, we can exploit the double formulation of the problem, as a Young SDE and a nonlinear YDE, to establish uniqueness under weaker regularity for $T^w b$ than that of Theorem~\ref{main thm2}. However, this comes at the price of prescribing some H\"older regularity for $w$, which might limit its regularising effect.

\begin{thm}\label{main thm4}
Let $\beta$ as above, $b\in C^\alpha_x$ for some $\alpha\in (0,1)$ and $w\in C^\delta_t$ a deterministic path with $H+\alpha\delta>1$; suppose that $T^w b\in C^{1/2}_t C^{\alpha+\nu}_x$ for some $\nu>0$ satisfying
\begin{equation}\label{main thm4 condition}
    \alpha + \nu (2H-1) > 1 + \frac{1}{2H}.
\end{equation}
Then for $\mu^H$-a.e. $\omega$ the following holds: for every $x_0\in \mathbb{R}^d$ there exists a unique solution to
\[
x_t = x_0 + \int_0^t b(x_s)\mathd \beta_s(\omega) +w_t
\]
in the class $x\in (w+C^{H-}_t)\cap C^\delta_t$, where the above integral is meaningful in the Young sense.
\end{thm}

The proofs of Theorems~\ref{main thm1}-\ref{main thm4} will be presented in Section~\ref{sec: proofs main results}; observe that they only rely on the analytical regularity of $T^w b$, where $w$ is a deterministic continuous path.
There is plenty of choice for $w$, as the next statements show.

\begin{cor}\label{main cor3}
Let $w$ be sampled as an fBm of parameter $\delta\in (0,1)$, $b$ be a compactly supported distribution of regularity $C^\alpha_x$, $\alpha\in\RR$, such that
\begin{equation}\label{main cor3 condition}
    \alpha>2-\frac{1}{\delta}\left(H-\frac{1}{2}\right).
\end{equation}
Then almost every realisation of $w$ satisfies condition~\eqref{main thm2 condition}. If in addition
\begin{equation}\label{main cor3 condition2}
    \alpha>n+1-\frac{1}{\delta}\left(H-\frac{1}{2}\right),
\end{equation}
then almost every realisation satisfies condition~\eqref{main thm2 condition2}. Moreover, under~\eqref{main cor3 condition} (resp. \eqref{main cor3 condition2}), generic $w\in C^\delta_t$ satisfy \eqref{main thm2 condition} (resp. \eqref{main thm2 condition2}), genericity being understood in the sense of prevalence. Finally, if $w$ is sampled as either a $p-\log$-Brownian motion or an infinite series of fBms (see Section~4 from~\cite{harang2020cinfinity}), then any choice of $\alpha\in\RR$ and $n\in\NN$ is allowed and we can drop the assumption of compact support on $b\in C^\alpha_x$.
\end{cor}

\begin{proof}
The case of $w$ sampled as an fBm follows from the results from~\cite{galeati2020noiseless}, see for instance Remark~7 or Section~3.3 more in general; indeed for $b$ as above, almost every realisation of $w$ satisfies
\[
T^w b\in C^{\frac{1}{2}}_t C^{\alpha+\nu}_x \quad \forall\ \nu<\frac{1}{2\delta}.
\]
Under condition~\eqref{main cor3 condition}, it's possible to find $\varepsilon>0$ small enough such that $\nu=1/(2\delta)-\varepsilon$ satisfies~\eqref{main thm2 condition}; similarly under condition~\eqref{main cor3 condition2}, we can choose $\nu=1/(2\delta)-\varepsilon$ so that~\eqref{main thm2 condition2} holds.
The conclusion follows from an application of Theorem~\ref{main thm2}.
The statement for generic $w\in C^\delta_t$ follows from the exact same reasoning, only applying Theorem~2 from~\cite{galeati2020noiseless} instead.
The last statement follows from the fact that these processes are infinitely regularising (see Section~4 from~\cite{harang2020cinfinity} for more details), so that $T^w b\in C^\alpha_t C^n_x$ for all $\alpha\in (0,1)$ and $n\in\NN$.
\end{proof}

\begin{rem}
The result shows that the introduction of a suitable perturbation $w$ allows to give meaning and solve the SDE with arbitrarily irregular distributional drift $b$; moreover the associated flow of solutions can become arbitrarily regular in space.
\end{rem}

\begin{cor}\label{main cor5}
Let $w$ be sampled as an fBm of parameter $\delta\in (0,1)$ such that $\delta+H<1$ and $b$ be a compactly supported distribution of regularity $C^\alpha_x$ such that
\begin{equation}\label{main cor5 condition}
    \alpha>\max\left\{\frac{1-H}{\delta}, 1+\frac{1}{2H}-\frac{1}{\delta}\left(H-\frac{1}{2}\right)\right\}.
\end{equation}
Then almost every realisation of $w$ satisfies the assumptions of Theorem~\ref{main thm4}. Moreover, under~\eqref{main cor5 condition}, generic $w\in C^\delta_t$ satisfy \eqref{main thm4 condition}, genericity being understood in the sense of prevalence.
\end{cor}

\begin{proof}
The proof is analogue to that of Corollary~\ref{main cor3}, only relying on Theorem~\ref{main thm4} instead.
Under condition~\ref{main cor5 condition}, $H+\alpha\delta>1$ and we can find $\nu=1/(2\delta)-\varepsilon$ with $\varepsilon>0$ sufficiently small such that~\eqref{main thm4 condition} holds.
The conclusion then follows from the results from~\cite{galeati2020noiseless} and Theorem~\ref{main thm4}.
\end{proof}

\begin{rem}
It can be checked that, in order for condition~\eqref{main cor5 condition} to be satisfied for some $\alpha<1$, it must be imposed $H>\sqrt{2}/2$. With a slight abuse, we can consider the fBm of parameter $H=1$ to be given by $\beta_t = N t$, where $N$ is a standard normal (this is the only possible $1$-self-similar centered Gaussian process); observe that in the limit $H\uparrow 1$ conditions~\eqref{main cor3 condition},~\eqref{main cor5 condition} become respectively
\[
\alpha>2-\frac{1}{2\delta},\qquad \alpha>\max\left\{0,\frac{3}{2}-\frac{1}{2\delta}\right\}
\]
which is consistent with the results from~\cite{Catellier2016} with $\mathd \beta_t$ replaced by $\mathd t$. 
\end{rem}

\subsection{Outline of the paper}
In Section \ref{sec:prelim on nonlinear integration and avg fields} we give a short overview of the existing theory on classical averaged fields and non-linear Young integration. In Section \ref{sec: avg fields w multiplicative noise} we investigate the multiplicative averaged field, both from an analytic and probabilistic point of view, and establish its space-time regularity. Section \ref{sec:young eq with multiplicative noise} deals with regularisation of SDEs by additive perturbations; several theorems regarding existence and uniqueness are given, as well as a discussion of the meaning of wellposedness of these random equations. Proofs of the main results from Section \ref{sec: main results} are given here. In Section \ref{sec: further extensions}, some elementary extensions  of the previous results are provided. We conclude in Section \ref{sec:concluding} with a discussion on open problems and future directions.

\subsection{Notation}
Below is a list of frequently used notation and conventions:
\begin{itemize}
    \item We denote by $C^\infty_c(\RR^d)$ the space of smooth compactly supported  functions and by $\cD(\RR^d)$ its dual.
    \item Similarly, $\cS(\RR^d)$ is the Schwartz space of rapidly decreasing functions on $\RR^d$, $\cS'(\RR^d)$ its dual.
    \item $\cB^\alpha_{p,q}$ denotes the classical in-homogeneous Besov spaces, for $\alpha\in \RR$, $p,q\in [1,\infty]$.
    \item We write  $C^\alpha_x:=\cB_{\infty,\infty}^{\alpha}(\RR^d)$; $C^n_b(\mathbb{R^d};\mathbb{R}^n)$ is the space of bounded, $n$-times differentiable maps with bounded derivatives up to order $n$. Their norms are denoted respectively by $\Vert\cdot\Vert_{\alpha}, \Vert\cdot\Vert_{C^n_b}$.
    \item Given a Banach space $E$, $\gamma\in (0,1)$, $C^\gamma_t E=C^\gamma([0,T];E)$ denotes the classical H\"older space of $E$-valued functions; we equip it with the H\"older seminorm and norm
    \begin{equation*}
        \llbracket f\rrbracket_{\gamma,E}:= \frac{\|f_{s,t}\|_E}{|t-s|^\gamma}, \quad
        \|f\|_{\gamma,E}=\|f(0)\|_E+\llbracket f\rrbracket_{\gamma,E},  
    \end{equation*}
    where we use the increment notation $f_{s,t}:=f(t)-f(s)$. %
    \item Of particular interest will be the choices $E=\mathbb{R}^d$, $E=C^\eta_x$ and $E=C^{\eta,\lambda}_x$, where $C^{\eta,\lambda}_x$ denotes a weighted H\"older space, see Definition \ref{defn: weighted holder}; they define the spaces $C^\gamma_t = C^\gamma_t \mathbb{R}^d$, $ C^\gamma_t C_x^\eta$ and $C^\gamma_t C^{\eta,\lambda}_x$. Their norms will be denoted respectively by $\Vert \cdot\Vert_\gamma$, $\Vert \cdot\Vert_{\gamma,\eta}$, $\Vert \cdot\Vert_{\gamma,\eta,\lambda}$.
    \item Whenever there is no possible ambiguity, we will keep using the shorthand notations $\Vert b\Vert_\alpha$, $\Vert \beta\Vert_H$, $\llbracket w\rrbracket_\delta$, $\Vert T^w b\Vert_{\gamma,\eta}$, $\Vert\Gamma^w b\Vert_{\gamma,\eta,\lambda}$, etc. 
    \item For $z\in \RR^d$, we define the translation operator $\tau$ acting on fields $b:\RR^d\rightarrow\RR^n$ by $\tau^z b=b(\cdot+z)$.
    \item Given a continuous path $w$, for any $\gamma\in (0,1)$, we set $w+C^\gamma_t:=\{w+g,\, g\in C^\gamma_t\}$.
    \item We denote by $B_R$ the open ball in $\RR^d$ centered at $0$ with radius $R>0$.
    \item Whenever a filtered probability space $(\Omega,\mathcal{F},\{\mathcal{F}_t\},\mathbb{P})$ appears, it is always assumed that $\mathcal{F}$ is $\mathbb{P}$-complete and that $\{\mathcal{F}_t\}$ satisfies the usual assumptions. We denote by $\mathbb{E}$ expectation with respect to $\mathbb{P}$.
\end{itemize}

\section{Preliminaries on averaging and nonlinear Young integration}\label{sec:prelim on nonlinear integration and avg fields}

\subsection{Properties of classical averaged fields}\label{sec2.1}

The averaged field $T^w b$ is by now a well studied object, see e.g. \cite{galeati2020noiseless,galeati2020prevalence,harang2020cinfinity,Catellier2016}; there is however not a unique way to define it and, depending on the situations, some definitions might be more practical than others.
For self-containedness, we provide here to the reader a brief overview of the topic, together with some of its properties which will be handy for later analysis. 
We start with an analytical definition of $T^w b$.

\begin{defn}[Averaging operator and averaged field] \label{def: avg field}
Let $w:[0,T]\rightarrow \RR^d $ be a measurable path and $E$ be a separable Banach space, continuously embedded in $\cS'(\RR^d)$, on which translations act isometrically, i.e. $\Vert \tau^v b\Vert_E= \Vert b\Vert_E$. We define the averaging operator $T^w$ as the continuous linear map from $E$ to $Lip([0,T],E)$ given by
\begin{equation*}
    T^w_tb=\int_0^t \tau^{w_s}b\dd s \qquad \forall\, t\in [0,T]. 
\end{equation*}
where the integral is meaningful in the Bochner sense. We will refer to $T^wb$ as an averaged field. 
\end{defn}
If $E\hookrightarrow C(\mathbb{R}^d)$, then the above definition corresponds to the pointwise one given by
\[
T^w_t b(x)=\int_0^t b(x+w_s)\dd s.
\]
If in addition $w$ is a continuous path, then it's easy to check that $T^w$ maps $C^\infty_c(\mathbb{R}^d)$ continuously into itself, allowing to define by duality $T^w$ on $\mathcal{D}(\mathbb{R}^d)$ by setting
\[
\langle T^w\varphi,\psi\rangle:= \langle \varphi,T^{-w}\psi\rangle \quad \forall\, \varphi\in \mathcal{D},\, \psi\in C^\infty_c.
\]
The main advantage of this definition is that it requires no underlying probability space and already allows to deduce some basic properties of the operators $T^w$.
\begin{lem}\label{lem: properties of classical averagaed field}
Let $w$ and $b$ be as in Definition \ref{def: avg field}. Then the following properties holds: 
\begin{itemize}
    \item[i.] Averaging and spatial differentiation commute, i.e. $\partial_iT^wb=T\partial_ib$ for all $i=1,\ldots,d$. 
    \item[ii.] Averaging and spatial convolution commutes, i.e. for any $K\in C^\infty_c(\RR^d)$, the following relation hold
    \begin{equation*}
        K\ast (T^wb)=T^w(K\ast b)=(T^wK)\ast b. 
    \end{equation*}
\end{itemize}
\end{lem}

We omit the proof, which can be found in Section~3.1 from~\cite{galeati2020noiseless}. Let us mention that Definition~\ref{def: avg field} is fairly elastic and allows to consider also time-dependent $b$; at the same time, its main drawback is that it doesn't allow to quantify the spatial regularity improvement of $T^w b$, compared to the original $b$, as an effect of the averaging procedure and the oscillatory nature of $w$. Nevertheless, if $T^w b$ is known to be regular, it provides efficient ways to approximate it.

\begin{lem}\label{lem: smooth approximation Twb}
Let $b\in E$ for some $E$ as above be such that $T^w b\in C^\gamma_tC^\alpha_x$ for some $\gamma\in (0,1]$ and $\alpha>0$, $(\rho^\eps)_{\eps>0}$ be a family of standard mollifiers and define $b^\eps:=\rho^\eps\ast b$. Then for any  $\delta>0$, $T^wb^\eps\rightarrow T^w b$ in $C^{\gamma-\delta}_tC^{\alpha-\delta}_x$ as $\eps\rightarrow 0$.
\end{lem}
\begin{proof}
The lemma is a slight improvement of Lemma~4 from~\cite{galeati2020noiseless}, the only difference being the claim that $T^wb^\eps\rightarrow T^w b$ in $C^{\gamma-\delta}_tC^{\alpha-\delta}_x$ \textit{globally} instead of just locally. As in~\cite{galeati2020noiseless}, thanks to the properties of averaging it holds
\begin{equation*}
    \Vert T^w b^\varepsilon\Vert_{\gamma,\alpha}
    = \Vert \rho^\varepsilon\ast T^w b\Vert_{\gamma,\alpha}
    \leq  \Vert  T^w b\Vert_{\gamma,\alpha}\quad \forall\, \varepsilon>0.
\end{equation*}
Moreover by properties of convolution, we have
\begin{equation*}
    \sup_{(t,x)\in[0,T]\times\RR^d}|(\rho^\varepsilon\ast T^w b) (t,x)-T^w b(t,x)|\lesssim \varepsilon^\alpha \Vert T^w b\Vert_{\gamma,\alpha}\to 0 \quad\text{ as }\varepsilon\to 0
\end{equation*}
i.e. uniform convergence holds. Standard interpolation estimates between the convergence in $C([0,T]\times\RR^d)$ and the uniform bound in $C^\gamma_t C^\alpha_x$ imply the conclusion.
\end{proof}

Another more probabilistic way to construct an averaged field is to consider a given distribution $b\in \mathcal{S}'(\mathbb{R}^d)$ and a continuous $\mathbb{R}^d$-valued stochastic process $(w_t)_{t\in [0,T]}$ on a probability space $(\Omega,\cF,\PP)$. Typically in this setting the goal is to show that $\mathbb{P}$-a.s. $T^w b$ is a well-defined, continuous random field, even if the original $b$ was not. We say that the process $w$ is $\rho$-regularising the distribution $b\in C^\alpha_x$ if $\mathbb{P}$-a.s. $T^w b\in C^\gamma_t C^{\alpha+\rho}_{x,loc}$ for some $\gamma>1/2$ and $\rho>0$.

In this sense, Gubinelli and Catellier proved in~\cite{Catellier2016} that if $b\in C^\alpha_x$ and $w$ is an fBm of parameter $H\in (0,1)$, then $w$ is $\rho$-regularising for any $\rho<1/(2H)$ (the results in~\cite{Catellier2016} actually also establish global estimates for $T^w b$, which require the introduction of suitable weighted H\"older norms similar to those in~\eqref{eq:def of weighted holder space}). Their results have then been extended to other classes of fields $b$, possibly of the form $b\in L^p_t C^\alpha_x$, in Section~7 from~\cite{le2020stochastic} and Section~3.3 from~\cite{galeati2020noiseless}.

Thus choosing a fBm with $H$ very small, the regularity of the associated averaged field $T^wb$ gets better. As the techniques used to prove the regularity of $T^w b$ are a probabilistic nature, the set of $\omega\in \Omega$ for which $T^{w(\omega)}b$ has the desired regularity depends on the given $b$ and cannot in general be chosen to be the same for all possible $b\in C^\alpha_x$. At the same time, it provides sharp estimates, which remarkably do not depend on the dimension of the ambient space $\mathbb{R}^d$.\\

A third approach, which combines analytic and probabilistic techniques, is based on the following observation: for any continuous path $w$, we have
\begin{equation}\label{sec2.1 eq occupation measure}
    T_{s,t}^w b(x) = b\ast \bar{\mu}^w_{s,t}(x), 
\end{equation}
where the measure $\bar{\mu}^w$ denotes the reflection of the occupation measure $\mu^w$, i.e. $\bar{\mu}^w_{s,t}(A):=\mu_{s,t}(-A)$ for any $A\in \cB(\RR^d)$. The occupation measure  $\mu^w$  associated to $w$ is defined as
\begin{equation*}
    \mu_t(A)=\lambda\{s\leq t|\, w_s\in A\}
\end{equation*}
for any Borel set $A\subset \RR^d$, where $\lambda$ denotes the Lebesgue measure on $[0,T]$. We say that $w$ admits a local time if $\mu^w$ is absolutely continuous w.r.t. the Lebesgue measure on $\mathbb{R}^d$, in which case the local time $L^w$ is exactly the density of $\mu^w$. Namely, it is the only non-negative element of $L^1(\mathbb{R}^d)$ such that
\begin{equation*}
    \mu_t(A)=\int_A L_t(z)\dd z \quad \forall\, \, A\in \mathcal{B}(\mathbb{R}^d).
\end{equation*}
In this case $T^w b= b\ast \bar{L}^w_t$ where $\bar{L}_t(x):=L_t(-x)$ and in order to show its regularity improvement, it suffice to establish the joint space-time regularity of the map $(t,x)\mapsto L^w_t(x)$. This line of approach was first explored in~\cite{Catellier2016}, via the notion of $\rho$-irregularity; the study of the joint space-time regularity of $L^w$ is however a topic of independent interest which has received a lot of attention, see~\cite{GemHoro} for a review.

It is shown in \cite[Thm. 17]{harang2020cinfinity} that if a Gaussian process $w:[0,T]\times \Omega \rightarrow \RR^d$ satisfies the following local nondeterminism condition for some $\zeta\in (0,2)$
\begin{equation*}
    \inf_{t>0} \inf_{s\in [0,t]} \inf_{z\in \RR^d;\, |z|=1} \frac{z^t {\rm Var}(w_t|\cF_s) z}{(t-s)^\zeta}>0,
\end{equation*}
then $\PP$-a.s. the local time $L^w$ is contained in the space $C^\gamma_t H^k$  for some 
\[
\gamma>\frac{1}{2},\quad k<\frac{1}{2\xi}-\frac{d}{2},
\]
where $H^k$ denotes the $L^2$-based Sobolev space. This result, combined with the relation \eqref{sec2.1 eq occupation measure}, allows to establish a regularising effect for all possible $b$ in a suitable class. Namely, if we denote by $\Omega'\subset \Omega$ the set of full measure where $L^w$ has the desired regularity, then by an application of Young's convolution inequality, we obtain that
\[ \Vert T^{w(\omega)} b\Vert_{C^\gamma_t C^{\beta+k}_x} \lesssim \Vert b\Vert_{H^\beta} \Vert L^{w(\omega)}\Vert_{C^\gamma_t H^k_x} \quad \forall\, b\in H^\beta_x\]
for all $\omega\in \Omega'$. In this case the regularity improvement holds on a set of full probability which is independent of the choice of $b\in H^\beta$. We can view $T^w$ as a (random) continuous linear operator from $H^\beta$ to $C^\gamma_t C^{k+\beta}_x$; in this sense we can call it an averaging operator.

The main drawback of this approach is that in general the regularity improvement will depend heavily on the dimension $d$ of the ambient space $\mathbb{R}^d$; for instance if $w$ is sampled as a Brownian motion, then its local time $L^w$ exists only for $d=1$, making the reasoning not applicable for $d\geq 2$. On the other hand, the aforementioned results for the averaged field $T^w b$ still provide a regularisation effect of order $\rho\sim 1$. For this reason in this article we will mostly refrain from considering the operator $T^w$, but rather only assume to be working with an averaged field $T^w b$ of suitable regularity.\\

Let us finally mention that in the papers \cite{galeati2020noiseless,galeati2020prevalence}, Gubinelli and one of the authors showed that the regularity properties of $T^w b$ (resp. $L^w$) in fact  hold for almost all continuous paths (in the sense of prevalence), see Theorem~1 from~\cite{galeati2020noiseless}. This largely speaks to the generality that is obtained through considerations of averaged fields in connection with ODEs, as in principle one does not impose any statistical assumption on the perturbation $w$. For instance, the results from~\cite{galeati2020noiseless} can be combined with our results, Theorems~\ref{main thm2} and~\ref{main thm4}, to deduce that generic perturbations $w$ regularise multiplicative SDEs driven by fBm,

\subsection{Non-linear Young integration and equations}\label{sec:Non-linear Young integration and equations}

We recall in this section some of the main results on the theory of abstract nonlinear Young differential equations, which is by now a well understood topic, see~\cite{Catellier2016,hu2017nonlinear,harang2020cinfinity,galeati2020noiseless}.

We start by introducing the class of vector fields $A:[0,T]\times\mathbb{R}^d\to\mathbb{R}^d$ we will work with; from now on, whenever $A$ appears, it will be implicitly assumed that $A(0,x)=0$ for all $x$. We also adopt the incremental notation $A_{s,t}(x)=A(t,x)-A(s,x)$.

\begin{defn}
We say that $f\in C(\mathbb{R}^d;\mathbb{R}^d)$ belongs to $C^\eta_{x,loc}$ for $\eta\in (0,1)$ if the following quantities are finite for any $R>0$:
\begin{equation*}
    \llbracket f\rrbracket_{\eta,R}:=\sup_{x, y\in B_R;\, x\neq y} \frac{| f(x)-f(y)|}{ |x-y|^\eta},
    \quad
    \Vert f\Vert_{\eta,R}:= \llbracket f\rrbracket_{\eta,R} + \sup_{x\in B_R} |f(x)|.
\end{equation*}
Given $A\in C([0,T]\times\mathbb{R}^d;\mathbb{R}^d)$, we say that $A\in C^\gamma_t C^\eta_{x,loc}$ for $\gamma,\eta\in (0,1)$ if similarly, for any $R>0$, it holds
\begin{equation*}
    \llbracket A\rrbracket_{\gamma,\eta,R}:=
    \sup_{0\leq s<t\leq T} \frac{\llbracket A_{s,t}\rrbracket_{\eta,R}}{|t-s|^\gamma}<\infty, 
    \quad
    \Vert A\Vert_{\gamma,\eta,R}:=
    \sup_{0\leq s<t\leq T} \frac{\Vert A_{s,t}\Vert_{\eta,R}}{|t-s|^\gamma}<\infty.
\end{equation*}
$A^n\to A$ in $C^\gamma_t C^\eta_{x,loc}$ if $\Vert A^n-A\Vert_{\gamma,\eta,R}\to 0$ as $n\to \infty$ for any $R\geq 0$; $A\in C^\gamma_t C^{n+\eta}_{x,loc}$ if $A$ admits spatial derivatives up to order $n$ and $D_x^k A\in C^\gamma_t C^\eta_{x,loc}$ for any $k\leq n$.\end{defn}

Given $A$ as above, we can define the non-linear Young integral of $A$ along a curve $\theta$.

\begin{thm}\label{secyoung thm existence young integral}
Let $A\in C^\gamma_t C^\eta_{x,loc}$ and $\theta\in C^\nu_t$ with $\gamma+\eta\nu>1$. Then the following limit exists and is independent of the choice of partitions $\cP$ of $[0,T]$ with infinitesimal mesh:
\begin{equation*}
\int_0^T A(\mathd u,\theta_u)=\lim_{|\Pi|\to 0} \sum_i A_{t_i,t_{i+1}}(\theta_{t_i})
\end{equation*}
We say that $\int_0^T A(\mathd u,\theta_u)$ is a {\em non-linear Young integral}.
More generally, the construction holds for any subinterval $[s,t]\subset [0,T]$ and allows to define a map $t\mapsto \int_0^t A(\mathd u,\theta_u)$ with the following properties:
\begin{itemize}
\item[i.] $\int_0^s A(\mathd u,\theta_u) + \int_s^t A(\mathd u,\theta_u)=\int_0^t A(\mathd u,\theta_u)$ for all $0\leq s\leq t\leq T$.
\item[ii.] $\int_0^\cdot A(\mathd u,\theta_u)\in C^\gamma_t$ and there exists a constant $C=C(\gamma,\gamma+\eta\nu,T)$ such that, taking $R=\Vert \theta\Vert_{\infty}$, it holds
\begin{equation*}
\Big| \int_s^t A(\mathd u,\theta_u) - A_{s,t}(\theta_s)\Big|
\leq C |t-s|^{\gamma+\eta\nu} \llbracket A\rrbracket_{\gamma,\beta,R} \llbracket \theta \rrbracket_\nu^\eta,
\end{equation*}
\begin{equation*}
\Big\Vert \int_0^\cdot A(\mathd u,\theta_u) \Big\Vert_{\gamma}
\leq C \Vert A\Vert_{\gamma,\eta,R} (1+\llbracket \theta \rrbracket_\nu^\eta).
\end{equation*}
\item[iii.] If in addition $\partial_t A$ exists and is continuous, then $\int_0^\cdot A(\mathd u,\theta_u) = \int_0^\cdot \partial_u A(u,\theta_u) \mathd u$.
\item[iv.] The map from $C^\gamma_t C^\eta_{x,loc}\times C^\nu_t \to C^\gamma_t$ given by $(A,\theta)\mapsto \int_0^\cdot A(\mathd u,\theta_u)$ is linear in $A$ and continuous in both variables (in the respective topologies).
\end{itemize}
\end{thm}

We can then pass to define the non-linear Young differential equation (YDE) associated to a drift $A\in C^\gamma_t C^\eta_{x,loc}$.  

\begin{defn} Let  $A$ be given as in Theorem \ref{secyoung thm existence young integral}.  We say that $\theta\in C^\nu_t$ is a solution starting at $\theta_0\in\RR^d$ to the nonlinear YDE
\begin{equation}\label{secyoung defn YDE 1}
\mathd \theta_t = A(\mathd t,\theta_t)
\end{equation}
if $\gamma+\eta\nu>1$ and $\theta$ satisfies
\begin{equation}\label{secyoung defn YDE 2}
\theta_t = \theta_0+\int_0^t A(\mathd u,\theta_u) \quad \forall\, t\in [0,T].
\end{equation}
\end{defn}

In order to provide a global solution theory, local bounds on $A$ are not enough and suitable growth conditions must be introduced.

\begin{defn}\label{defn: weighted holder}
For $\eta,\lambda\in (0,1)$, we define the weighted H\"older space $C^{\eta,\lambda}_x=C^{\eta,\lambda}(\mathbb{R}^d;\mathbb{R}^d)$ as the collection of all fields $f\in C^\eta_{x,loc}$ such that
\begin{equation*}
    \Vert f \Vert_{\eta,\lambda} := |f(0)|+\sup_{R\geq 1}  R^{-\lambda}\, \llbracket f\rrbracket_{\eta,R}<\infty.
\end{equation*}
$C^{\eta,\lambda}_x$ is a Banach space with the norm $\Vert\cdot \Vert_{\eta,\lambda}$; similar definitions hold for $C^{n+eta,\lambda}_x$, $n\in \NN$.
\end{defn}

\begin{defn}
We say that $A\in C^\gamma_t C^\eta_x$ if it satisfies global bounds, namely if
\begin{equation*}
    \llbracket A \rrbracket_{\gamma,\eta}:=
    \sup_{0\leq s<t\leq T} \frac{\llbracket A_{s,t}\rrbracket_\eta}{|t-s|^\gamma}<\infty,
    \quad
    \Vert A\Vert_{\gamma,\eta}:=
    \sup_{0\leq s<t\leq T} \frac{\Vert A_{s,t}\Vert_\eta}{|t-s|^\gamma}<\infty.
\end{equation*}
where $\llbracket\cdot\rrbracket_\eta$, $\Vert \cdot\Vert_{\eta}$ denote the classical Besov-H\"older seminorm and norm of $C^\eta(\mathbb{R}^d;\mathbb{R}^d)$ respectively. Similarly, $A\in C^\gamma_t C^{\eta,\lambda}_x$ for $\gamma,\eta,\lambda\in (0,1)$ if
\begin{equation*}
    \Vert A\Vert_{\gamma,\eta,\lambda}:=
    \sup_{0\leq s<t\leq T} \frac{\Vert A_{s,t}\Vert_{\eta,\lambda}}{|t-s|^\gamma}<\infty.
\end{equation*}
Observe that $C^\gamma_t C^{\eta,\lambda}_x$ is a Banach space endowed with the norm $\Vert \cdot\Vert_{\gamma,\eta,\lambda}$. The definitions for $C^\gamma_t C^{n+\eta}_x$ and $C^\gamma_t C^{n+\eta,\lambda}_x$ are analogue.
\end{defn}

\begin{rem}
Although the quantities $\Vert \cdot\Vert_{\gamma,\eta,R}$ and  $\Vert \cdot\Vert_{\gamma,\eta,\lambda}$ are related, since the latter measures how the first grows as a function of $R$, we ask the reader to keep in mind that they represents two different quantities.
Throughout the text $R\geq 0$ will always denote the radius of a ball $B(0,R)\subset \RR^d$ centered at zero, and so $\Vert \cdot\Vert_{\gamma,\eta,R}$ denotes the H\"older norm restricted to $[0,T]\times B(0,R)$; instead the parameter $\lambda\in (0,1)$ will be consistently used in relation to the weighted H\"older space $C^{\eta,\lambda}_x$. We believe that the exact meaning of the norm will always be clear from the context. 
\end{rem}

Observe that for $A\in C^\gamma_t C^{\eta,\lambda}_x$ we have an upper bound on the growth of $A_{s,t}$ at infinity. Indeed, for any $x\in\mathbb{R}^d$ such that $|x|\geq 1$, it holds
\[ |A_{s,t}(x)| \leq |A_{s,t}(x)-A_{s,t}(0)| + |A_{s,t}(0)|
\leq \llbracket A \rrbracket_{\gamma,\eta,\lambda} |t-s|^\gamma |x|^{\eta+\lambda} +\Vert A\Vert_{\gamma,\eta,\lambda} |t-s|^\gamma.
\]
In particular, if $\eta+\lambda\leq 1$, then $A_{s,t}$ has at most linear growth.\\

The following theorem gives sufficient conditions for well-posedness of the YDE associated to $A$, as well as existence and regularity of the associated flow.

\begin{thm}\label{secyoung thm existence and uniquness of YDE}
Suppose $A\in C^\gamma_t C^{\eta,\lambda}_x$ for some $\gamma, \eta,\lambda\in (0,1)$ such that $\gamma>1/2$, $\gamma(1+\eta)>1$ and $\eta+\lambda\leq 1$. Then for any $\theta_0\in \RR^d$ there exists a solution $\theta\in C^\gamma_t$ to the YDE \eqref{secyoung defn YDE 1} starting from $\theta_0$, as well as a constant $C=C(\gamma,\eta,T)$ such that
\begin{equation}\label{secyoung a priori estimate}
\Vert \theta\Vert_\gamma \leq C \exp( C \Vert A\Vert^2_{\gamma,\eta,\lambda}) (1+|\theta_0|).
\end{equation}
If $A\in C^\gamma_t C^{\eta,\lambda}_x\cap C^\gamma_t C^{1+\eta}_{x,loc}$, such solution is unique and the YDE admits a $C^\gamma_t C^1_{x,loc}$ flow.
Finally, if $A\in C^\gamma_t C^{n+\eta}_{x,loc}$, then the flow belongs to $C^\gamma_t C^n_{x,loc}$.
\end{thm}

\begin{proof}
The existence of a global solution under the condition $C^\gamma_t C^{\eta,\lambda}_x$, together with the a priori estimate~\eqref{secyoung a priori estimate}, 
follows from Theorem~3.1 from~\cite{hu2017nonlinear} (see also Theorem~2.9 from~\cite{Catellier2016}). Since estimate~\eqref{secyoung a priori estimate} is uniform over all possible $\theta_0$ in a bounded ball, we can apply localization arguments (see Remark~2.10 and Section~2.3 from~\cite{Catellier2016}, as well as Remark~14 from~\cite{galeati2020noiseless}) and assume wlog $A\in C^\gamma_t C^{1+\eta}_x$ (resp. $C^\gamma_t C^{n+\eta}_x$); uniqueness and $C^\gamma_t C^1_x$-regularity of the flow are then consequences of Theorem~3.5 from~\cite{hu2017nonlinear} (see also Theorems~16 and 17 from~\cite{galeati2020noiseless} or Proposition~28 from~\cite{harang2020cinfinity}). Finally, higher regularity follows from Theorem~2 from~\cite{harang2020cinfinity} of equivalently Theorem~18 from~\cite{galeati2020noiseless}.
\end{proof}

In order to compare solutions associated to different data $(\theta_0,A)$, a general methodology based on Comparison Principles was introduced in~\cite{Catellier2016}. The version given here is based on Theorem~9 from~\cite{galeati2020noiseless}.

\begin{thm}\label{secyoung comparison theorem}
Let $R, M>0$, $A^i\in C^\gamma_t C^{1+\eta,\lambda}_x$ for some $\gamma,\eta,\lambda$ as in Theorem \ref{secyoung thm existence and uniquness of YDE}. Suppose $\Vert A^i\Vert_{\gamma,1+\eta,\lambda}\leq M$, $\vert \theta_0^i\vert\leq R$ for $i=1,2$, and denote by $\theta^i$ the unique solution associated to $(A^i,\theta_0^i)$. Then there exists a constant $C=C(\gamma,\eta,T,R,M)$, increasing in the last two variables, such that
\begin{equation}\label{secyoung estimate comparison theorem}
\Vert \theta^1-\theta^2\Vert_\gamma \leq C\big(|\theta_0^1-\theta_0^2| + \Vert A^1-A^2\Vert_{\gamma,1+\eta,\lambda}\big).
\end{equation}
\end{thm}

\begin{proof}
We only sketch the proof as it is almost identical to the one of Theorem~9 from~\cite{galeati2020noiseless}. Thanks to the a priori bound~\eqref{secyoung a priori estimate}, we can localize everything and assume $A^i\in C^\gamma_t C^{1+\eta}_x$ (the localization will produce constants depending on $R$ and $M$ which are incorporated in the final $C$).  It follows from Lemma~6 in~\cite{galeati2020noiseless} that $v:=\theta^1-\theta^2$ satisfies an affine classical YDE of the form
\[ v_t=v_0+\int_0^t v_s\cdot \mathd V_s+\psi_t\]
where
\begin{equation*}
    V_t =\int_0^1 \int_0^t \nabla_x A^1(\mathd s, \theta^2_s +\lambda (\theta^1_s-\theta^2_s)) \mathd \lambda,
    \quad \psi_t = \int_0^t (A^1-A^2)(\mathd s,\theta^2_s).
\end{equation*}
Standard estimates for solutions to affine Young equations are known, see for instance Lemma~19 from~\cite{galeati2020noiseless} or Section~6.2 from~\cite{lejay2010controlled}; by points {\em i.} and {\em ii.} of Theorem \ref{secyoung thm existence young integral}, we can estimate $\psi$ by
\[
\Vert \psi \Vert_\gamma \lesssim \Vert A^1-A^2\Vert_{\gamma,\eta} (1+\llbracket \theta^2\rrbracket_\gamma) \lesssim \Vert A^1-A^2\Vert_{\gamma,\eta}
\]
and the conclusion follows.
\end{proof}

As a nice corollary, we deduce continuous dependence of the flow $\Phi$ on the drift $A$.

\begin{cor}\label{cor: continuous dependence flow map}
Define a map $\mathcal{I}$ on $C^\gamma_t C^{1+\eta,\lambda}_x$ by $A\mapsto \mathcal{I}(A)$, where $\mathcal{I}(A)$ is the flow associated to $A$. Then $\mathcal{I}$ is a continuous map from $C^\gamma_t C^{1+\eta,\lambda}_x$ to $C([0,T]\times\mathbb{R}^d;\mathbb{R}^d)$, the latter being endowed with the topology of uniform convergence on compact sets. As a consequence, to any random field $A$ as above, we can associate a unique random flow $\Phi=\mathcal{I}(A)$.  
\end{cor}

\begin{proof}
The statement is an immediate consequence of estimate~\eqref{secyoung estimate comparison theorem}. Indeed, given $A^i\in C^\gamma_t C^{1+\eta,\lambda}_x$ with $\Vert A^i \Vert_{\gamma,1+\eta,\lambda}\leq M$, the solutions $\theta^i$ associated to $(A^i,\theta_0,)$ correspond to $\theta^i_t = \mathcal{I}(A^i)(t,\theta_0)$ and therefore from~\eqref{secyoung estimate comparison theorem} we deduce that
\begin{align*}
    \sup_{\theta_0\in B_R, t\in [0,T]} \vert \mathcal{I}(A^1)(t,\theta_0) -\mathcal{I}(A^1)(t,\theta_0) \vert
    & \leq \sup_{\theta_0\in B_R}  \Vert \mathcal{I}(A^1)(\cdot,\theta_0) - \mathcal{I}(A^2)(\cdot,\theta_0)\Vert_\gamma\\
    & \leq C \Vert A^1-A^2\Vert_{\gamma,1+\eta,\lambda}.
\end{align*}
Given a sequence $A^n\to A$ in $C^\gamma_t C^{1+\eta,\lambda}_x$, it must be bounded in $C^\gamma_t C^{1+\eta,\lambda}_x$ and therefore for any $R>0$ we can find $C_R>0$ such that
\begin{equation*}
     \sup_{\theta_0\in B_R, t\in [0,T]} \vert \mathcal{I}(A^n)(t,\theta_0) -\mathcal{I}(A)(t,\theta_0) \vert
     \leq C_R \Vert A^n-A\Vert_{\gamma,1+\eta,\lambda}\to 0
\end{equation*}
which shows uniform convergence on compact sets of $\mathcal{I}(A^n)$ to $\mathcal{I}(A)$. The last statement follows from the fact that continuous image of measurable functions is still measurable.
\end{proof}

\begin{rem}
The results from \cite{Catellier2016,hu2017nonlinear,galeati2020noiseless,harang2020cinfinity} actually show that, given a bounded family $\{A_n\}_n$ in $C^\gamma_t C^{1+\eta,\lambda}_x$, the associated flows $\mathcal{I}(A_n)$ are bounded in $C^\gamma_t C^1_{x,loc}$ (in the sense that all seminorms $\Vert \mathcal{I}(A_n)\Vert_{\gamma,1,R}$ are controlled). Thus interpolation estimates allow to improve the previous result by showing that, if $A_n\to A$ in $C^\gamma_t C^{1+\eta,\lambda}_x$, then $\mathcal{I}(A_n)\to\mathcal{I}(A)$ in $C^{\gamma-\varepsilon}_t C^{1-\varepsilon}_{x,loc}$ for any $\varepsilon>0$.
\end{rem}

%

%%%%%%%%%%

\section{Averaged fields with multiplicative noise}\label{sec: avg fields w multiplicative noise}
An averaged field with multiplicative noise is formally given by
\begin{equation}\label{eq:formal avg field}
    \Gamma^w_{s,t}b(x)=\int_s^t b(x+w_r)\dd \beta_r,\qquad x\in \RR^d,\,\,[s,t]\subset [0,T], 
\end{equation}
where we consider in general $w\in C([0,T];\mathbb{R}^d)$, $b\in \cD(\RR^d;\mathbb{R}^{d\times m})$ and $\beta\in C^H([0,T];\mathbb{R}^m)$ to be a H\"older continuous path with $H>1/2$.

The main goal of this Section is to prove the following result, which allows to rigorously construct $\Gamma^w b$ as a random field and to relate its space-time regularity to that of the classical averaged field $T^w b$.

\begin{thm}\label{thm: main result sec3}
Let $\beta=\{\beta_t\}_{t\in [0,T]}$ be a fBm of Hurst parameter $H>1/2$, with values in $\mathbb{R}^m$, defined on a probability space $(\Omega,\mathcal{F},\mathbb{P})$. Then for any deterministic $b\in\mathcal{S}(\mathbb{R}^d;\mathbb{R}^{d\times m})$ and $w\in C^\delta([0,T];\mathbb{R}^d)$ with $H+\delta>1$, it's possible to define the averaged field $\Gamma^w b$ in~\eqref{eq:formal avg field} pathwise as a Young integral; $\Gamma^w b$ can be regarded as a random field from $[0,T]\times \mathbb{R}^d$ to $\mathbb{R}^d$.\\
The definition extends continuously in a unique way to any pair $(b,w)$ with $b\in \mathcal{D}(\mathbb{R}^d;\mathbb{R}^{d\times m})$, $w\in C([0,T];\mathbb{R}^d)$ such that $T^w b\in C^\gamma_t C^\eta_x$ for some $\gamma>1-H$, $\eta\in(0,1)$.
In that case
\[
\Gamma^w b \in L^p(\Omega; C^{\gamma'}_t C^{\eta',\lambda}_x)
\quad \forall\, p<\infty, \,\gamma'<\gamma+H-1, \, \eta'<\eta, \, \lambda>0.
\]
and there exists $C>0$ (depending on all the above parameters) such that for any $(b^i,w^i)$ as above it holds
\begin{equation}\label{eq: main thm sec3 estimate}
    \EE\big[\, \big\Vert \Gamma^{w^1} b^1-\Gamma^{w^2} b^2\big\Vert_{\gamma',\eta',\lambda}^p\big]\leq C\, \big\Vert T^{w^1}b^1-T^{w^2} b^2\big\Vert^p_{\gamma,\eta}.
\end{equation}
More generally, estimate~\eqref{eq: main thm sec3 estimate} holds replacing $\eta'$, $\eta$ with $n+\eta'$, $n+\eta$ respectively, for any $n\in\mathbb{N}$; namely, $\Gamma^w b$ inherits higher space regularity from $T^w b$.
\end{thm}
\begin{rem}\label{rem: sufficent regularity requirement} Observe that in the statement above, if $T^w b\in C^\gamma_t C^2_x$ for some $\gamma$ such that $\gamma+H>3/2$, it is always possible to choose $\gamma'$, $\eta'$ and $\lambda$ such that $\gamma'>1/2$, $\gamma'(1+\eta')>1$ and $\eta'+\lambda<1$. 
\end{rem}

The proof of Theorem~\ref{thm: main result sec3} is presented throughout the section, which is structure as follows.

We first consider the more regular case in which $w\in C^\delta_t$ with $\delta+H>1$. Here we can give a rigorous analytical construction of the operator $b\mapsto \Gamma^w b$, as a map from $\cD(\RR^d)$ into itself; in this case, the definition does not require $\beta$ to be sampled as an fBm and instead holds for any given $H$-H\"older continuous path.

Next we restrict our attention to the fBm case, in which by more probabilistic techniques we can extend the definition of $\Gamma^w b$ to a larger class of $(w,b)$; this class is defined only in terms of the regularity of the classical averaged field $T^w b$.
A key point will be the use of a lemma from \cite{hairer2019averaging} to obtain suitable $L^p(\Omega)$ bounds for $\Gamma^wb$, combined with a modified version of the Garsia-Rodemich-Rumsey Lemma.
\subsection{Definition of averaging operator}\label{sec3.1}
The purpose of this section is to analitically define the multiplicative averaging operator $\Gamma^w$ as a map from $\mathcal{D}(\RR^d)$ to itself; to this end, we need to impose some regularity on $w$ and $\beta$, namely require $H+\delta>1$. The advantage of this approach is that the definition can be applied to any path $\beta\in C^H_t$, not necessarily sampled as an fBm; however we will see in the next sections that, in the fBm case, we can drop the condition $H+\delta>1$, by defining $\Gamma^w b$ as a random field.\\
Recall that for any $v\in\RR^d$, $\tau^v$ denotes the translation operator by $v$, i.e. $\tau^v b(\cdot)=b(\cdot+v)$. 

\begin{lem}\label{lem defn distributional averaging}
Let $\alpha\in\RR$, $w\in C^\delta_t$, $\beta\in C^H_t$ and $\eta\in (0,1]$ such that
\begin{equation*}
H+\eta\delta>1.
\end{equation*}
Then for any $b\in C^{\alpha+\eta}_x$ there exists a unique element of $C^H_t C^\alpha_x$, which we denote by $\Gamma^w b$ and which we will refer to as a multiplicative averaged field, such that
\begin{equation*}
||\Gamma^w_{s,t} b-b(\cdot+w_s)\beta_{s,t}||_\alpha\lesssim |t-s|^{H+\eta\delta}.
\end{equation*}
Moreover there exists a constant $C=C(H+\eta\delta,T)$ such that for any $b\in C^{\alpha+\eta}_x$ it holds
\begin{equation}\label{sectemp bound young averaging}
\Vert \Gamma^w b\Vert_{H,\alpha} \leq C \Vert b\Vert_{\alpha+\eta} \llbracket \beta \rrbracket_H (1+\llbracket w\rrbracket_\delta).
\end{equation}
In particular, the map $\Gamma^w:b\mapsto \Gamma^w b$ is an element of $\mathcal{L}(C^{\alpha+\eta}_x;C^H_t C^\alpha_x)$. If $\alpha>0$, then $\Gamma^w b$ defined as above coincides with the pointwise map defined by the Young integral
\begin{equation}\label{sectemp pointwise averaging definition}
(\Gamma^w_{s,t} b)(x)=\int_s^t b(x+w_r)\mathd \beta_r.
\end{equation}
\end{lem}
\begin{proof}
All the statements easily follow from an application of the sewing lemma (e.g. \cite[Lemma 4.2]{Friz2014}). Set, for any $s\leq t$, $\Xi_{s,t}:=(\tau^{w_s}b) \beta_{s,t}\in C^\alpha_x$; it holds $\delta \Xi_{s,u,t}=(\tau^{w_s}b-\tau^{w_u}b) \beta_{s,t}$ with the estimates
\begin{align*}
\Vert \delta \Xi_{s,u,t}\Vert_\alpha
& = \Vert \tau^{w_s}b-\tau^{w_u}b \Vert_\alpha\, |\beta_{s,t}|
\lesssim \Vert b \Vert_{\alpha+\eta} |w_{s,u}|^\eta |\beta_{s,t}|\\
& \leq \Vert b \Vert_{\alpha+\eta}
 \llbracket w \rrbracket^{\eta}_\delta \llbracket \beta \rrbracket_H |t-s|^{H+\delta\eta},
\end{align*}
where we used the basic estimate
\begin{equation}\label{eq: baby bernstein}
\Vert \tau^y b-\tau^z b\Vert_\alpha \lesssim |y-z|^\eta \Vert b\Vert_{\alpha+\eta}.
\end{equation}
To see~\eqref{eq: baby bernstein}, observe that by Bernstein estimates, for any Littlewood-Paley block of $b$ it holds
\begin{equation*}
\Vert \tau^y \Delta_n b-\tau^z \Delta_n b\Vert_\infty
\lesssim \Vert \Delta_n b\Vert_\infty,\quad \Vert \tau^y \Delta_n b-\tau^z \Delta_n b\Vert_\infty
\lesssim 2^n |y-z| \Vert \Delta_n b\Vert_\infty,
\end{equation*}
which interpolated together provide, for any $\eta\in [0,1]$,
\begin{equation*}
\Vert \tau^y b-\tau^z b\Vert_\alpha
= \sup_n \{ 2^{n\alpha} \Vert \tau^y \Delta_n b-\tau^z \Delta_n b\Vert_\infty\}
\lesssim |y-z|^\eta \sup_n \{ 2^{n(\alpha+\eta)} \Vert \Delta_n b\Vert_\infty\}
= |y-z|^\eta \Vert b\Vert_{\alpha+\eta}.
\end{equation*}   
\noindent The sewing lemma thus implies the existence and uniqueness of $\Gamma^w b$, as well as the bound
\begin{equation*}
\Vert \Gamma^w_{s,t} b - b(\cdot+w_s)\beta_{s,t}\Vert_\alpha
\lesssim \Vert b \Vert_{\alpha+\eta}
 \llbracket w \rrbracket^{\eta}_\delta \llbracket \beta \rrbracket_H.
\end{equation*}
We then have
\begin{align*}
\Vert \Gamma^w_{s,t} b\Vert _\alpha
& \leq \Vert \tau^{w_s} b\Vert_\alpha |\beta_{s,t}| + C \Vert b \Vert_{\alpha+\eta} \llbracket w \rrbracket^{\eta}_\delta \llbracket \beta \rrbracket_H |t-s|^{H+\eta\delta}\\
& \lesssim_T |t-s|^H \Vert b\Vert_{\alpha+\eta} \llbracket \beta \rrbracket_H (1+\llbracket w \rrbracket_\delta),
\end{align*}
which implies bound~\eqref{sectemp bound young averaging}. The last claim follows from the fact that the Young integral in~\eqref{sectemp pointwise averaging definition} corresponds to the sewing of $\langle \Xi_{s,t},\delta_x\rangle$ and thus must coincide with $\langle \Gamma_{s,t}^w b,\delta_x\rangle$. 
\end{proof}
The operator $\Gamma^w$ behaves similarly to the classical averaging operator $T^w$; we summarize some of its properties in the following two lemmas. 
\begin{lem}\label{lem properties averaging} Let $\Gamma^wb$ be given as in Lemma \ref{lem defn distributional averaging}. Then  the following  properties hold:
\begin{itemize}
\item[i.] Averaging and space differentiation (in the distributional sense) commute:
\[ \partial_{x_i} \Gamma^w b = \Gamma^w \partial_i b\quad \forall\, b\in C^\alpha_x,\, i=1,\ldots,d.\]
\item[ii.] Averaging and spatial convolution commute: for any $\varphi\in C^\infty_c$ it holds
\[ \varphi\ast (\Gamma^w b) = \Gamma^w (\varphi\ast b)\quad \forall\, b\in C^\alpha_x.\]
\item[iii.] If $b$ is compactly supported, then so is $\Gamma^w b$, with $\mathrm{supp}\, \Gamma^w_{s,t} b\subset \mathrm{supp}\, b + B(0,\Vert w\Vert_{\infty})$ for all $s,t$. Similarly, if $b^1$ and $b^2$ coincide on $B(0,R)$, then $\Gamma^w b^1$ and $\Gamma^w b^2$  coincide on $B(0,R-\Vert w\Vert_\infty)$.
\item[iv.] The operator $\Gamma^w$ can be extended to an operator from $\mathcal{D}(\RR^d)$ to itself by the duality formula
\[ \langle \Gamma_{s,t}^w\psi,\varphi \rangle:=\langle \psi, \Gamma_{s,t}^{-w}\varphi\rangle\quad \forall \psi\in \mathcal{D}(\RR^d),\,\varphi\in C^\infty_c(\RR^d).  \]
\end{itemize}
\end{lem}
\begin{proof}
The proof is analogue to that of Lemma \ref{lem defn distributional averaging}. Indeed, by setting $\Xi[b]_{s,t}:=(\tau^{w_s} b) \beta_{s,t}$, it is immediate to check that
\[ \partial_{x_i} \Xi[b] = \Xi[ \partial_{x_i} b ],\quad \varphi\ast \Xi_{s,t}[b]=\Xi_{s,t}[\varphi\ast b] \]
and so the same relations must hold between the respective sewings, proving points~\textit{i.} and~\textit{ii.}. The first part of point~\textit{iii.} follows from the fact that, for any $s<t$, $\Xi_{s,t}[b]$ is supported on $\mathrm{supp}\, b + B(0,w_s)\subset \mathrm{supp}\, b + B(0,\Vert w\Vert_\infty)$ and the second part by applying a similar reasoning to $b^1-b^2$. Finally, it follows from Lemma~\ref{lem defn distributional averaging} and point~\textit{iii.} that $\Gamma_{s,t}^w$ continuously maps $C^\infty_c$ into itself; therefore also the dual definition from $\mathcal{D}(\mathbb{R}^d)$ to itself is meaningful. Whenever $\psi$ and $\varphi$ are both smooth, we have the relation
\[ \langle (\tau^{w_s} \psi)\beta_{s,t}, \varphi\rangle = \langle \psi, (\tau^{-w_s}\varphi)\beta_{s,t}\rangle \]
which implies the same relation for the respective sewings, i.e. $\langle \Gamma^w_{s,t} \psi, \varphi\rangle = \langle \psi, \Gamma^{-w}_{s,t}\varphi\rangle$.
\end{proof}
\begin{lem}\label{lem approximation averaging}
Let $b\in\mathcal{D}(\RR^d)$ be such that $\Gamma^w b\in C^\gamma_t C^{\alpha,\lambda}_x$ for some $\gamma,\lambda\in (0,1)$ and $\alpha\in (0,\infty)$. Let $\{\rho^\varepsilon\}_{\varepsilon>0}$ be a family of standard mollifiers and set $b^\varepsilon=\rho^\varepsilon\ast b$. Then for any $\varepsilon>0$ it holds $\Gamma^w b^\varepsilon\in C^\gamma_t C^{\alpha,\lambda}_x$ with
\begin{equation}\label{lem approximation bound}
\Vert \Gamma^w b^\varepsilon\Vert_{\gamma,\alpha,\lambda}
\lesssim  \Vert \Gamma^w b\Vert_{\gamma,\alpha,\lambda};
\end{equation}
moreover $\Gamma^w b^\varepsilon \to \Gamma^w b$ as $\varepsilon\to 0$ in $C^{\gamma'}_t C^{\alpha',\lambda}_x$ for any $\gamma'<\gamma$ and $\alpha'<\alpha$.
\end{lem}
\begin{proof}
It is enough to prove the claim for $\alpha\in (0,1)$, as the other cases follow by repeating the same argument for $D^k \Gamma^w b = \Gamma^w D^k b$.
The bound~\eqref{lem approximation bound} follows from point~\textit{iii.} of Lemma~\ref{lem properties averaging}, since we have
\[
\Vert \Gamma^w b^\varepsilon\Vert_{\gamma,\alpha,R}
= \Vert \rho^\varepsilon\ast\Gamma^w b \Vert_{\gamma,\alpha,R}
\lesssim \Vert \Gamma^w b \Vert_{\gamma,\alpha,R+\varepsilon}
\lesssim R^\lambda \Vert \Gamma^w b \Vert_{\gamma,\alpha,\lambda}
\]
where we used the fact that $\rho^\varepsilon$ is supported in $B_\varepsilon$ and $(R+\varepsilon)^\lambda\sim R^\lambda$ since $R\geq 1$ and $\varepsilon\in (0,1)$. By properties of convolutions, it holds
\[
\sup_{(t,x)\in [0,T]\times B_R} |\Gamma^w b^\varepsilon(t,x)-\Gamma^w b(t,x)|
\lesssim \varepsilon^\alpha \Vert \Gamma^w b\Vert_{\gamma,\alpha,R+\varepsilon}
\lesssim \varepsilon^\alpha R^\lambda \Vert \Gamma^w b\Vert_{\gamma,\alpha,\lambda};
\]
Interpolating this estimate with the uniform bound~\eqref{lem approximation bound}, we obtain that for any $\theta\in (0,1)$ it holds
\[
\Vert \Gamma^w b^\varepsilon-\Gamma^w b\Vert_{\theta\gamma, \theta\alpha,\lambda}
= \sup_{R\geq 1} \{ R^{-\lambda} \Vert \Gamma^w b^\varepsilon-\Gamma^w b\Vert_{\theta\gamma, \theta\alpha,\lambda} \}
\lesssim \varepsilon^{(1-\theta)\alpha} \Vert \Gamma^w b\Vert_{\gamma,\alpha,\lambda}\to 0 \quad \text{ as }\varepsilon\to 0.
\]
By the arbitrariness of $\theta\in (0,1)$ we can conclude.
\end{proof}

\subsection{$L^p$ bounds for averaging operators with multiplicative fBm in the smooth case}

We will now assume that $\{\beta_t\}_{t\in [0,T]}$ is sampled as a fractional Brownian motion with $H>1/2$, with trajectories in $C_t^{H-}$; observe that all the results from the previous section still apply with $H$ replaced by $H-\varepsilon$, $\varepsilon$ sufficiently small. Through probabilistic techniques, we will show that we can extend the definition of $\Gamma^w b$ to other choices of $b$ and $w$ and that $\Gamma^w b$ inherits the spatial regularity of $T^w b$ (at least locally). To this end, we will use a probabilistic inequality for integration with respect to a fractional Brownian motion with $H>\frac{1}{2}$ proven by Hairer and Li \cite[Prop. 3.4]{hairer2019averaging}. We recite this result in the following proposition.  

\begin{prop}\label{prop: lemma 3.4 HairerLi}
Let $\beta:[0,T]\times \Omega\rightarrow \RR^m$ be a fractional Brownian motion with Hurst parameter $H>1/2$, $f:[0,T]\rightarrow \RR$ be a $\mathcal{F}_0$-adapted process. 
Furthermore, assume that for some $\gamma>1/2$ with $H+\gamma>1$ it holds  $\| \int_0^\cdot f_r\dd r\|_\gamma\in L^q(\Omega)$ for some $q>2$. Then for any $p\in [2,q)$ there exists a constant $C= C(p,q,\gamma,H,T)$ such that
\begin{equation*}
    \Big\Vert\int_s^t f_r\dd \beta_r\Big\Vert_{L^p(\Omega)}
    \leq
    C\, \mathbb{E}\bigg[ \Big\Vert \int_0^\cdot f_r\dd r \Big\Vert_\gamma^q\bigg]^{\frac{1}{q}}\,|t-s|^{H+\gamma-1}
    \quad \forall\, [s,t]\subset [0,T].  
\end{equation*}
\end{prop}
\begin{rem}
The rather elegant point of the above lemma is that it extends the class of integrands with respect to fBm to distributions $f\in \cS'(\mathbb{R})$ such that $\int_0^\cdot f_r\dd r \in  C^\gamma_t$ for some $\gamma>1/2$. It immediately extends to the case $f\in \cS'(\mathbb{R};\mathbb{R}^{d\times m})$ by reasoning component-wise. Keeping in mind that our interest is in averaging operators, by setting $f_r=\tau^{w_r}b(x)$ for some continuous path $w$, $\int_s^t f_r \dd \beta_r$ is a well defined random variable in $L^p(\Omega)$ as long as $\int_0^\cdot \tau^{w_r} b(x)\dd r = \int_0^\cdot b(x+w_r)\dd r = T^w b(\cdot,x)$ belongs to $C^\gamma_t$.
\end{rem}

\begin{lem}\label{lem: Lp bounds for mult. avg. Op.}
Let  $b\in \mathcal{S}(\mathbb{R}^d;\mathbb{R}^{d\times m})$, $\beta$ be an fBm of parameter $H>1/2$ and $w\in C^\delta_t$ a deterministic path such that $H+\delta>1$.
Define the multiplicative averaged field $\Gamma^w b$ pathwise as in the previous section; namely, for any $\omega\in \Omega$ such that $\beta(\omega)\in C^{H-}_t$, set
\begin{equation}\label{eq: m avg field}
    \Gamma^w_{s,t}b(x)(\omega) := \int_s^t b(x+w_r)\dd \beta_r(\omega). 
\end{equation}
Then for any $p\geq 2$ and $\gamma>1-H$ we have the following estimates:
\begin{equation*}
\begin{aligned}
{\rm \textit{i.}}& \qquad     \|\Gamma_{s,t}^w b(x)\|_{L^p(\Omega)}
\lesssim \| T^w b\|_{\gamma,\eta} |t-s|^{H+\gamma-1},
    \\
 {\rm \textit{ii.}}& \qquad   \|\Gamma_{s,t}^w b(x)-\Gamma_{s,t}^w b(y)\|_{L^p(\Omega)}
 \lesssim \| T^w b\|_{\gamma,\eta} |x-y|^\eta|t-s|^{H+\gamma-1},
    \\
   {\rm \textit{iii.}}& \qquad  \| \nabla\Gamma_{s,t}^w b(x)- \nabla\Gamma_{s,t}^w b(y)\|_{L^p(\Omega)}
   \lesssim \| T^w b\|_{\gamma,1+\eta} |x-y|^{\eta}|t-s|^{H+\gamma-1}.
\end{aligned}
\end{equation*}

\end{lem}

\begin{proof}
The results are a direct application of Proposition~\ref{prop: lemma 3.4 HairerLi}. It follows from the assumptions that $\Gamma^w b\in C^{H-}_t C^\alpha_x$, as well as $T^w b\in C^\gamma_t C^\alpha_x$, for any $\alpha\in\mathbb{R}$ and $\gamma\in [0,1]$; for any $p\geq 2$ it holds
\begin{align*}
   \Vert \Gamma^w_{s,t} b(x) \Vert_{L^p(\Omega)}
   & = \Big\Vert \int_s^t b (x+w_r)\dd \beta_r \Big\Vert_{L^p(\Omega)}\\
   & \lesssim
   \Big\Vert \int_0^\cdot b (x+w_r) \dd r\Big\Vert_{\gamma} |t-s|^{H+\gamma-1}\\
   & \sim \Vert T^w b(\cdot,x)\Vert_{\gamma} |t-s|^{H+\gamma-1}, 
\end{align*}
which implies that point~\textit{i.} holds. Similarly, for any $x,y\in\mathbb{R}^d$ we have
\begin{align*}
     \|\Gamma_{s,t}^w b(x)-\Gamma_{s,t}^w b(y)\|_{L^p(\Omega)}
     & \lesssim \Vert T^w b(\cdot,x)-T^w b(\cdot,y)\Vert_\gamma |t-s|^{H+\gamma-1}\\
     & \lesssim \Vert T^w b\Vert_{\gamma,\eta} |x-y|^\eta |t-s|^{H+\gamma-1}.
\end{align*}
Point~\textit{iii.} follows from the fact that $\nabla \Gamma^w b= \Gamma^w \nabla b$ and an application of points~\textit{i.} and~\textit{ii.} with $b$ replaced by $\nabla b$.

\end{proof}

In order to provide a control on the joint space-time regularity of $\Gamma^w b$ in terms of that of $T^w b$, we need to combine Lemma~\ref{lem: Lp bounds for mult. avg. Op.} with a suitable modification of the classical Garsia-Rodemich-Rumsey (GRR) Lemma; %
a direct application of the results from~\cite{garsia1970real} is not enough, as it only provides local estimates, while
the theory outlined in Section~\ref{sec:Non-linear Young integration and equations} requires the additional growth condition $\Gamma^w b\in C^\gamma_t C^{1+\eta,\lambda}_x$.

Recall that for general $A:[0,T]\times \mathbb{R}^d\to \mathbb{R}^d$ it holds
\[
\Vert A\Vert_{\gamma,\eta,\lambda} \lesssim \llbracket A \rrbracket_{\gamma,\eta,\lambda} + \Vert A(\cdot,0)\Vert_\gamma
\]
where by definition
\[
\llbracket A\rrbracket_{\gamma,\eta,\lambda} = \sup_{0\leq s<t\leq T} \frac{\llbracket A_{s,t} \rrbracket_{\eta,\lambda}}{|t-s|^\gamma},
\]
and we recall that for $f:\mathbb{R}^d\to\mathbb{R}^d$, the weighted H\"older seminorm is given by
\begin{equation}\label{eq:def of weighted holder space}
\llbracket f \rrbracket_{\eta,\lambda}
:=\sup_{R\geq 1} R^{-\lambda}\, \llbracket f \rrbracket_{\eta,\lambda,R}
= \sup_{R\geq 1}\, \sup_{x,y\in B_R;\, x\neq y} \frac{|f(x)-f(y)|}{R^\lambda\, |x-y|^\eta}.
\end{equation}
In order to establish $C^\gamma_t C^{\eta,\lambda}_x$-regularity of random fields, we need the following lemma.
%slight modification of the classical GRR lemma.
%
\begin{lem}\label{lem GRR}
Let $\{A(t,x):t\in  [0,T],x\in\RR^d \}$ be a family of $\RR^d$-valued random variables satisfying the following condition for some $\kappa>0$ and $p\geq 1$:
\begin{equation}\label{lem GRR hypothesis}
\EE [|A_{s,t}(x)-A_{s,t}(y)|^p] \leq \kappa |t-s|^{1+\beta_1} |x-y|^{d+\beta_2} \quad \forall\, 0\leq s\leq t\leq T,\, x,y\in\RR^d.
\end{equation}
Then for any $\gamma,\eta,\lambda\in (0,1)$ such that
\begin{equation*}
\gamma<\frac{\beta_1}{p},\quad \eta<\frac{\beta_2}{p},\quad \lambda>\frac{\beta_2+d}{p}-\eta,
\end{equation*}
there exists a constant $C=C(\eta,\gamma,\lambda,\beta_1,\beta_2,p,d)$ and a continuous modification of $A$ such that
\begin{equation}\label{lem GRR conclusion}
\EE\Big[ \llbracket A \rrbracket_{\gamma,\eta,\lambda}^p\Big]\leq C\,\kappa.
\end{equation}

\end{lem}

\begin{proof}
Existence of a jointly continuous modification of $A$ which is locally H\"older continuous follows from classical application of GRR lemma, so we only need to focus on estimate~\eqref{lem GRR conclusion}. We can assume $A$ to take values in $\RR$, as the general case follows reasoning componentwise. We will first prove the following claim: if $b$ is a continuous random field such that
\[ \EE[|b(x)-b(y)|^p]\leq \kappa |x-y|^{d+\beta} \quad \forall\,x,y\in \RR^d,
\]
then for any $\eta<\beta/p$ and $\lambda$ such that $\eta+\lambda<(\beta+d)/p$, then  $b\in C^{\eta,\lambda}_x$ and there exists a constant $c_1=c_1(d,p,\eta,\beta)$ such that
\begin{equation}\label{lem GRR claim}
\EE \left[\llbracket b\rrbracket_{\eta,\lambda}^p\right] \leq c_1\, \kappa.
\end{equation}
Indeed by the classical GRR lemma, for any continuous function $f$, there exists a constant $c_2=c_2(d,\eta,\beta,p)$ which is independent of $R$ such that
\[ \llbracket f\rrbracket_{\eta,R}^p = \bigg(  \sup_{x,y\in B_R;\, x\neq y} \frac{|f(x)-f(y)|}{|x-y|^\eta} \bigg)^p \leq c_2 \int_{B_R\times B_R} \frac{|f(x)-f(y)|^p}{|x-y|^{2d+\eta p}} \,\mathd x \mathd y.
\]
Applied to the field $b$, this implies that for any $R> 0$ it holds
\[ \EE\left[ R^{-\lambda p} \llbracket b \rrbracket_{\eta,R}^p\right]
\leq c_2\, \kappa\, R^{-\lambda p} \int_{B_R\times B_R} |x-y|^{\beta-\alpha p-d} \,\mathd x\mathd y = c_1\, \kappa\,  R^{\beta+d-\eta p-\lambda p}.
  \]
for any $\eta<\beta/p$. Now consider the sequence $R=2^n$ with $n\in \NN$, then
\begin{align*}
\EE\left[ \bigg( \sup_{R=2^n, n\in\NN} R^{-\lambda} \llbracket b \rrbracket_{\eta,R}\bigg)^p\right]
\leq \EE\left[ \sum_{R=2^n} R^{-\lambda p}  \llbracket b \rrbracket_{\eta,R}^p\right]
\leq c_1\, \kappa\, \sum_n 2^{n(\beta+d-\eta p-\lambda p)}
\leq c_3\, \kappa
\end{align*}
for some $c_3=c_3(d,\eta,\beta,\lambda,p)$, under the condition $\beta+d-\eta p-\lambda p<0$. Finally, for any $R\geq 1$, choosing $n\in\NN$ such that $2^n \leq R<2^{n+1}$, it holds
\begin{equation*}
R^{-\lambda} \llbracket b\rrbracket_{\eta,R}
\leq R^{-\lambda} \llbracket b\rrbracket_{\eta,2^{n+1}}
\leq R^{-\lambda}\, 2^{\lambda(n+1)} \sup_{r=2^m, m\in\NN} r^{-\lambda} \llbracket b \rrbracket_{\eta,r}
\leq 2^\lambda \sup_{r=2^m, m\in\NN} r^{-\lambda} \llbracket b \rrbracket_{\eta,r}
\end{equation*}
which combined with the previous estimates implies the claim~\eqref{lem GRR claim}. In order to conclude, observe that for any $s\leq t$, applying the above to $b=A_{s,t}$, by hypothesis~\eqref{lem GRR hypothesis} we obtain
\begin{equation*}
\EE\left[\llbracket A_{s,t}\rrbracket_{\eta,\lambda}^p\right] \leq c_1\, \kappa\, |t-s|^{1+\beta_1}
\end{equation*}
and the conclusion follows by applying classical Kolmogorov continuity criterion.
\end{proof}

\subsection{Proof of Theorem~\ref{thm: main result sec3} }
We now have all the ingredients to complete the proof of the main result of this section. We start by showing that estimate~\eqref{eq: main thm sec3 estimate} is true when $b$ and $w$ are taken sufficiently regular.

\begin{lem}\label{cor: regualrity of multiplcative field}
Let $b^1$, $b^2$, $w^1$, $w^2$, $\beta$ be as in Lemma~\ref{lem: Lp bounds for mult. avg. Op.}, $\gamma>1-H$ and $\eta\in (0,1)$ fixed parameters. Then for any choice of $(p,\gamma',\eta',\lambda)$ such that
\[
p\geq 2,\quad \gamma'<\gamma+H-1,\quad \eta'<\eta,\quad \lambda>0,
\]
there exists a constant $C$ (which depends on $d$, $T$ and the parameters above) such that
\begin{equation}\label{estimate regularity multiplicative averaging}
    \EE\big[\, \big\Vert \Gamma^{w^1} b^1-\Gamma^{w^2} b^2\big\Vert_{\gamma',\eta',\lambda}^p \big]
    \leq C\, \big\Vert T^{w^1}b^1-T^{w^2} b^2\big\Vert^p_{\gamma,\eta}.
\end{equation}
\end{lem}
\begin{proof}
As the multiplicative averaging acts linearly, it suffices to show the statement for a single $T^w b$ as above. Interpolating the bounds \textit{i.-ii.} of Lemma~\ref{lem: Lp bounds for mult. avg. Op.}, we see that for any $\theta\in [0,1]$ it holds
\begin{equation*}
    \|\Gamma_{s,t}^wb(x)-\Gamma_{s,t}^wb(y)\|_{L^p(\Omega)} \lesssim \Vert T^w b\Vert_{\gamma,\eta} |t-s|^{H+\gamma-1}|x-y|^{\theta\eta}\qquad \forall\,p\geq 2,\,\,  x,y\in \RR^d.
\end{equation*}
Therefore $\Gamma_{s,t}^w b$ satisfies condition \eqref{lem GRR hypothesis} for the choice $\beta_1=p(H+\gamma-1)-1$ and $\beta_2=p\theta\eta-d$; since $p$ can be chosen arbitrarily large, we conclude by Lemma \ref{lem GRR} that for any 
$$
\gamma'<H+\gamma-1,\qquad \eta'<\theta\eta,\qquad \lambda>\eta(1-\theta) , 
$$
it holds
\[ \EE\big[\, \Vert \Gamma^w b\Vert_{\gamma',\eta',\lambda}^p\big]\leq C \Vert T^w b \Vert_{\gamma,\eta}^p; \]
observe that we can take $\theta$ arbitrarily close to $1$, so that $\eta'$ is arbItrarily close to $\eta$ and $\lambda$ is arbitrarily small.
\end{proof}

\begin{proof}[Proof of Theorem~\ref{thm: main result sec3}]
The proof is divided in two natural steps: we will first show that, thanks to Lemma~\ref{cor: regualrity of multiplcative field}, we can extend the definition of $\Gamma^w b$ to the case of regular $b$ and continuous (but not necessarily H\"older regular) $w$; then we will show that, under the assumption that $T^w b$ is sufficiently regular, the definition further extends to the case of distributional $b$.

\textit{Step 1.} Let $b\in C^1_b$, $\{w^n\}_n$ be a sequence in $C^\delta_t$, with $\delta+H>1$, such that $w^n\to w$ uniformly on $[0,T]$. Our aim is to show that the sequence $\Gamma^{w^n} b$ is Cauchy in a suitable weighted H\"older space and thus admits a unique limit, which we define to be $\Gamma^w b$. In particular, while we cannot define anymore the field $\Gamma^w b$ analytically as done in Section~\ref{sec3.1}, it is still well defined as a random variable.

Since $b\in C^1_b$, for any $n,m\in\NN$ we have the estimates
\begin{equation*}
    \left\lvert \int_s^t b(x+w^n_r) \mathd r-\int_s^t b(x+w^m_r) \mathd r\right\rvert 
    \leq \int_s^t \Vert b\Vert_{C^1_b} |w^n_r-w^m_r| \mathd r
    \leq \Vert w^n-w^m\Vert_\infty \Vert b\Vert_{C^1_b} |t-s|
\end{equation*}
and similarly, for fixed $n$ and any $x,y\in\RR^d$,
\begin{equation*}
    \left\lvert \int_s^t b(x+w^n_r) \mathd r-\int_s^t b(y+w^n_r) \mathd r\right\rvert 
    \leq | x-y| \Vert b\Vert_{C^1_x} |t-s|.
\end{equation*}
One can then apply triangular inequality and interpolate the two inequalities above to deduce that, for any $\eta\in (0,1)$, it holds
\begin{align*}
    \left\lvert T^{w^n}_{s,t} b(x)- T^{w^m}_{s,t}b(y) \right\rvert
    \lesssim \Vert b\Vert_{C^1_b} |x-y|^\eta \Vert w^n-w^m\Vert_{\infty}^{1-\eta} |t-s|.
\end{align*}
Since $w^n\to w$ uniformly in $[0,T]$, the sequence $\{w^n\}_n$ is Cauchy, and by the above estimate so is $\{T^{w^n}b\}_n$ in $C^\gamma_t C^{\eta}_x$, for any $\gamma,\eta<1$. Combined with~\eqref{estimate regularity multiplicative averaging}, this implies that for any $\gamma'<H$, $\eta'<\eta$, $\lambda>0$ and $p\in [2,\infty)$ it holds
\begin{equation*}
     \EE\big[ \big\Vert \Gamma^{w^n} b-\Gamma^{w^m} b\big\Vert_{\gamma',\eta',\lambda}^p\big]
     \lesssim \big\Vert T^{w^n}b-T^{w^m} b\big\Vert^p_{1,\eta'+\varepsilon}
     \lesssim \Vert b\Vert_{C^1_b} \Vert w^n -w^m\Vert_{\infty}^{1-\eta'-\varepsilon},
\end{equation*}   
where we chose $\varepsilon>0$ s.t. $\eta'+\varepsilon<1$. Therefore the sequence $\{\Gamma^{w^n}b\}_n$ is Cauchy in $L^p(\Omega; C^{\gamma'}_t C^{\eta',\lambda}_x)$ and it admits a unique limit, which we define to be $\Gamma^w b$. It follows from the estimates above that this is a good definition, as it does not depend on the chosen sequence $\{w_n\}_n$ such that $w_n\to w$.

More generally, by iterating the reasoning to $D^k b$ for $k\leq n$, the above procedure shows that if $b\in C^{n+1}_x$ and $w$ is a continuous path, then $\Gamma^w b$ belongs to $C^{\gamma'}_t C^{n+\eta',\lambda}_x$. 
By construction, inequality~\eqref{estimate regularity multiplicative averaging} still holds for any pairs $(w^i,b^i)$ with $w^i\in C^0_t$ and $b^i\in C^1_b$.

\textit{Step 2.} We now want to pass to the case in which $b$ is distributional, $w$ is continuous and $T^w b\in C^\gamma_t C^\eta_x$ (resp. $C^\gamma_t C^{n+\eta}_x$) for some $\gamma>1-H$.

By Lemma~\ref{lem: smooth approximation Twb} we can choose a family of mollifiers $\{\rho^\varepsilon\}_{\varepsilon>0}$, a parameter $\delta>0$ arbitrarily small and a sequence $\varepsilon_n\to 0$ such that setting $b_n=b^{\varepsilon_n}=\rho_{\varepsilon_n}\ast b$, it holds $T^w b_n\to T^w b$ in $C^{\gamma-\delta}_t C^{\eta-\delta}_x $. In particular, $\{T^w b_n\}_n$ is a Cauchy sequence in $C^{\gamma-\delta}_t C^{\eta-\delta}_x $ and choosing $\delta$ such that $\gamma-\delta>1-H$, by the previous step $\{\Gamma^w b_n\}_n$ are well defined random fields; moreover for any $\gamma'<\gamma+H-\delta-1$, $\eta'<\eta-\delta$, $\lambda>0$ and $p\in [2,\infty)$ they satisfy
\begin{equation*}
     \EE[ \big\Vert \Gamma^{w} b_n-\Gamma^{w} b_m\big\Vert_{\gamma',\eta',\lambda}^p]
     \lesssim \big\Vert T^w b_n-T^w b_m \big\Vert^p_{\gamma-\delta,\eta-\delta}.
\end{equation*}
This implies that $\{\Gamma^w b_n\}_n$ is a Cauchy sequence in $L^p (\Omega; C^{\gamma'}_t C^{1+\eta',\lambda}_x)$ and thus admits a unique limit, which we define to be $\Gamma^w b$. It is clear from Lemma~\ref{lem: smooth approximation Twb} that $\Gamma^w b$ does not depend on the chosen family of mollifiers; more generally the above estimates imply that for any sequence of smooth functions $b^n$ s.t. $T^w b^n\to T^w b$ in $C^{\gamma-\delta}_t C^{\eta-\delta}_x$, the associated multiplicative averaged fields $\Gamma^w b_n$ must converge to $\Gamma^w b$. Moreover for any pair of random fields $\Gamma^{w_1} b_1$, $\Gamma^{w_2} b_2$ defined in this way, for $w^i$ continuous paths and $b_i$ possibly distributional fields, we have the inequality
\begin{equation*}
     \EE\big[\, \big\Vert \Gamma^{w_1} b_1-\Gamma^{w_2} b_2\big\Vert_{\gamma',\eta',\lambda}^p\big]
     \lesssim \big\Vert T^{w_1} b_1-T^{w_2} b_2 \big\Vert^p_{\gamma,\eta}.
\end{equation*}
which can be rephrased as the fact that the multiplicative averaging, seen as a map $T^w b\mapsto \Gamma^w b$ from $C^\gamma_t C^{\eta}_x$ to $L^p(\Omega; C^{\gamma'}_t C^{\eta',\lambda}_x)$, is linear and continuous.

The general case of $T^w b\in C^\gamma_t C^{n+\eta}_x$ follows as before by iterating the reasoning to the derivatives $D^k T^w b= T^w D^k b$.
\end{proof}

\begin{rem}
If $w\in C^\delta_t$ with $\delta+H>1$, the above procedure is consistent with the one from Section~\ref{sec3.1}, namely the random field $\Gamma^w b$ is a regular representative of the random distribution defined pathwise by means of Lemma~\ref{lem defn distributional averaging}.
\end{rem}

\begin{rem}
Several properties satisfied by the analytical definition of $\Gamma^w b$ from Lemma~\ref{lem properties averaging} extend by the approximation procedure to the more general definition of Theorem~\ref{thm: main result sec3}, once they are interpreted as equalities between random variables. For instance it is still true that, $K\in C^\infty_c$, $K\ast \Gamma^w b=\Gamma^w (K\ast b)$; similarly, if both $T^w b$ and $T^w \nabla b$ are regular enough, then $\Gamma^w \nabla b=\nabla \Gamma^w b$.
\end{rem}

\begin{rem}\label{rem: approximation mult avg}
The proof of Theorem~\ref{thm: main result sec3} also contains the following fact: if $T^w b\in C^\gamma_t C^{n+\eta}_x$, then it's possible to find a sequence $(b^n,w^n)$ with $b^n\in C^\infty_x$, $w^n\in C^1_t$ such that $b^n\to b$ in the sense of distributions, $w^n\to w$ in the uniform convergence and $\Gamma^{w^n} b^n\to \Gamma^w b$ in $L^p(\Omega;C^{\gamma'}_t C^{n+\eta',\lambda}_t)$ for any $\gamma'<\gamma+H-1$, $\eta'<\eta$ and $\lambda>0$.
\end{rem}

\section{Regularisation of SDEs by additive perturbations}\label{sec:young eq with multiplicative noise}

We are now ready to prove the regularizing effect of certain paths on SDEs with multiplicative noise. Towards this aim, we begin to motivate this section by showing that when $b$ is a smooth vector field, $w\in C^\delta_t$, and $t\mapsto \beta_t$ is a sample path of a fractional Brownian motion with $H\in (\frac{1}{2},1)$ such that $\delta+H>1$,  then  multiplicative SDEs formally given by
\begin{equation}
    \dd x_t=b(x_t) \dd \beta_t+\dd w_t, \qquad x_0\in \RR^d
\end{equation}
can be solved in the non-linear Young equations framework, outlined in section \ref{sec:Non-linear Young integration and equations}.  Just as in the non-multiplicative case, these results can then be generalised to allow for distributional drifts $b$, still under the assumption that $\delta+H>1$. These solutions preserves the natural notion of a pathwise solution, in the sense that if $\{b^n\}_{n}$ is a sequence of smooth functions approximating the distribution $b$ in a suitable distribution space, then the corresponding solutions $x^n\rightarrow x$ in $C^\delta_t$.

\subsection{Classical YDEs as averaged equations}\label{subsec:YDE with multiplicative noise}

The content of this section, similarly to that of Section~\ref{sec3.1}, is entirely analytic and holds even when $\beta$ is not sampled as an fBm but rather a given deterministic function.
For notational simplicity, we consider $\beta\in C^H_t$, but all the statements generalize to the case $\beta\in C^{H-}_t$, as the conditions on $H$ are always in the form of a strict inequality.

Let us briefly recall the setting: here $b\in \mathcal{D}(\mathbb{R}^d;\mathbb{R}^{d\times m})$ (mostly regular for the moment), $w\in C^\delta([0,T];\mathbb{R}^d)$ and $\beta\in C^H([0,T];\mathbb{R}^m)$; we look for a solution $x\in C([0,T];\mathbb{R}^d)$.

We start by showing that the nonlinear YDE formulation of the problem is a natural generalisation of the original one, whenever $b$ and $w$ are sufficiently regular.

\begin{prop}\label{sectemp prop smooth case}
Let $b\in C^2_b$, $w\in C^\delta_t$ and $\beta\in C^H_t$ with $H>1/2$, $H+\delta>1$. Then for any $x_0\in\RR^d$ there exists a unique solution $x\in C^\delta_t$ to the perturbed Young differential equation
\begin{equation}\label{sectemp YDE}
x_t = x_0 +\int_0^t b(x_s)\mathd\beta_s+w_t\quad\forall\,t\in [0,T];
\end{equation}
in particular,  $x=\theta+w$, where $\theta\in C^H_t$ is the unique solution to the nonlinear YDE
\begin{equation}\label{sectemp nonlinear YDE}
\theta_t =\theta_0 +\int_0^t \Gamma^wb (\mathd s,\theta_s).
\end{equation}
For any $\alpha\in(0,1)$ satisfying $H+\alpha\delta>1$ there exists a constant $C=C(\alpha,\delta,H,T)$ such that $\theta$ satisfies the a priori estimate 
\begin{equation}\label{sectemp smooth a priori bound 1}
\llbracket \theta\rrbracket_H \leq C (1+\Vert b\Vert_\alpha^2\, \llbracket\beta \rrbracket_H^2)(1+\llbracket w\rrbracket_\delta).
\end{equation}
\end{prop}
\begin{proof}
It is easy to check that $x\in C^\delta_t$ solves~\eqref{sectemp YDE} iff $\theta=x-w\in C^\delta_t$ satisfies
\[
\theta_t
= \theta_0 + \int_0^t b(\theta_s+w_s)\dd \beta_s =\theta_0 + \int_0^t \tilde{b}(s,\theta_s)\dd \beta_s\qquad \forall\, t\in [0,T]
\]
where $\tilde{b}(t,z):=b(z+w_t)$; by properties of Young integrals, any such $\theta$ must also belong to $C^H_t$. The drift $\tilde{b}$ satisfies
\begin{equation*}
\big\vert \tilde{b}(t,z_1)-\tilde{b}(s,z_2)\big\vert
+ \big\vert \nabla\tilde{b}(t,z_1)-\nabla \tilde{b}(s,z_2)\big\vert
\lesssim \Vert b\Vert_{C^2_b} |z_1-z_2|+\Vert b\Vert_{C^2_b} \llbracket w\rrbracket_\delta |t-s|^\delta
\end{equation*}
which by classical results implies existence and uniqueness of solutions to the YDE associated to $\tilde{b}$ in the class $C^H_t$, see for instance Theorem~2.1 from~\cite{rascanu2002differential} or Section~3 from~\cite{cong2018nonautonomous}.\\
In order to show that $\theta$ solves~\eqref{sectemp nonlinear YDE}, it is enough to prove that $\int_0^\cdot b(w_s+\theta_s)\mathd\beta_s =\int_0^\cdot \Gamma^w b (\mathd s,\theta_s)$. Since $b\in C^2_b$ and $H+\delta>1$, by Lemma~\ref{lem defn distributional averaging} we have $\Gamma^w b\in C^H_t C^1_x$ and the nonlinear Young integral $\int_0^\cdot \Gamma^w b (\mathd s,\theta_s)$ is well defined (because $\theta\in C^H_t$ and $H>1/2$). By the respective definition of the two integrals, it holds
\begin{align*}
\bigg\vert \int_s^t b(w_r+\theta_r) & \mathd\beta_r -\int_s^t \Gamma^w b(\mathd r,\theta_s)\bigg\vert\\
& = \bigg\vert \int_s^t b(w_r+\theta_r)\mathd\beta_r \pm b(w_s+\theta_s)\beta_{s,t} \pm \Gamma^w_{s,t}b (\theta_s) -\int_s^t \Gamma^w b(\mathd r,\theta_s)\bigg\vert\\
& \lesssim |t-s|^{H+\delta} +\bigg\vert b(w_s+\theta_s)\beta_{s,t}-\int_s^t b(\theta_s+w_r)\mathd \beta_r\bigg\vert
\lesssim |t-s|^{H+\delta}
\end{align*}
which implies that they must coincide.

We now move on to prove \eqref{sectemp smooth a priori bound 1}.
For any $0<\Delta<T$, denote by $\llbracket \theta\rrbracket_{H,\Delta}$ (resp. $\llbracket \theta\rrbracket_{\delta,\Delta}$) the quantity
\[ \llbracket \theta\rrbracket_{H,\Delta} = \sup_{|t-s|\leq\Delta} \frac{\vert\theta_{s,t}\vert}{|t-s|^H}.\]
By properties of Young integrals, for any $s<t$ such that $|t-s|<\Delta$ it holds
\begin{align*}
\vert \theta_{s,t}\vert
& = \Big\vert \int_s^t b(w_r+\theta_r)\mathd \beta_r\Big\vert\\
& \lesssim \vert b(w_s+\theta_s)\beta_{s,t}\vert + |t-s|^{H+\alpha\delta} \llbracket b \rrbracket_\alpha \llbracket \beta\rrbracket_H \llbracket \theta+w\rrbracket^\alpha_{\delta,\Delta}\\
& \lesssim |t-s|^H \Vert b\Vert_\alpha \llbracket \beta\rrbracket_H + |t-s|^H \Delta^{\alpha\delta}\,  \llbracket b \rrbracket_\alpha \llbracket \beta\rrbracket_H (1+\llbracket w \rrbracket_\delta+\llbracket \theta\rrbracket_{\delta,\Delta})\\
& \lesssim |t-s|^H \Vert b\Vert_\alpha \llbracket \beta \rrbracket_H \big(1+ \Delta^{\alpha\delta} + \Delta^{\alpha\delta} \llbracket w\rrbracket_\delta \big) + |t-s|^H \Delta^{\alpha\delta} \Vert b\Vert_\alpha \llbracket \beta \rrbracket_H \llbracket \theta \rrbracket_{H,\Delta}.
\end{align*}
Dividing by $|t-s|^H$, taking the supremum over $|t-s|\leq \Delta$, we find $\kappa=\kappa(\alpha,\delta,H,T)$ s.t.
\[
\llbracket\theta\rrbracket_{H,\Delta} \leq \kappa \Vert b\Vert_\alpha \llbracket \beta \rrbracket_H \big(1+ \Delta^{\alpha\delta} + \Delta^{\alpha\delta} \llbracket w\rrbracket_\delta \big) +\kappa \Delta^{\alpha\delta} \Vert b\Vert_\alpha \llbracket \beta \rrbracket_H \llbracket \theta \rrbracket_{H,\Delta};
\]
choosing $\Delta$ such that $\kappa \Delta^{\alpha\delta} \Vert b\Vert_{C^\alpha_x} \llbracket \beta \rrbracket_{C^H_t} \leq 1/2$, $\kappa \Delta^{\alpha\delta} \Vert b\Vert_{C^\alpha_x} \llbracket \beta \rrbracket_{C^H_t} \sim 1$ we obtain
\begin{equation*}
\llbracket \theta \rrbracket_{H,\Delta}
\lesssim 1+\Vert b\Vert_\alpha \llbracket \beta \rrbracket_H + \llbracket w \rrbracket_\delta.
\end{equation*}
Applying Exercise~4.24 from~\cite{Friz2014} we deduce
\begin{equation*}
    \llbracket\theta\rrbracket_H
    \lesssim \Delta^{H-1} (1+\Vert b\Vert_\alpha \llbracket \beta \rrbracket_H + \llbracket w \rrbracket_\delta)\\
    \lesssim (\Vert b\Vert_\alpha \llbracket \beta \rrbracket_H)^{\frac{1-H}{\alpha\delta}}\, (1+\Vert b\Vert_\alpha \llbracket \beta \rrbracket_H + \llbracket w \rrbracket_\delta)
\end{equation*}
and the conclusion follows from the fact that $(1-H)/(\alpha\delta)<1$ by hypothesis.
\end{proof}

\subsection{General YDEs as averaged equations}

In the case $b$ is regular enough for the classical YDE~\eqref{sectemp YDE} to be meaningful, the nonlinear Young formalism still gives non trivial criteria in order to establish uniqueness of solutions, as the next proposition shows.
\begin{prop}\label{sectemp prop uniqueness mildly regular case}
Let $b\in C^\alpha_x$ for some $\alpha\in (0,1)$ such that $H+\alpha\delta>1$. Then for any $x_0\in \RR^d$ there exists at least one solution $x\in C^\delta_t$, $x\in w+C^H_t$ to the YDE~\eqref{sectemp YDE}. If $\Gamma^w b\in C^\gamma_t C^{1+\eta}_{x,loc}$ for some $\gamma,\eta\in (0,1)$ satisfying
\[\gamma+\eta H>1,\]
then such solution $x$ is unique in the class $w+C^H_t$.
\end{prop}
\begin{proof}
The proof follows a similar reasoning to those from Section~4.1 of \cite{galeati2020noiseless}, so we will mostly sketch it.

\textit{Step 1: Existence.} Let $b^\varepsilon$ be a sequence of mollifications of $b$ and denote by $x^\varepsilon$ the unique solution of the YDE~\eqref{sectemp YDE} associated to $b^\varepsilon$ with initial data $x_0$. Then $x^\varepsilon=\theta^\varepsilon+w$ satisfy the a priori bound~\eqref{sectemp smooth a priori bound 1}, uniformly in $\varepsilon>0$ and so by Ascoli--Arzel\`a we can extract a subsequence $\theta^{\varepsilon_n}$ such that $\theta^{\varepsilon_n}\to \theta$ in $C^{H'}_t$ for any $H'<H$. Combining this fact with $b^{\varepsilon_n}\to b$ in $C^{\alpha'}_x$ for any $\alpha'<\alpha$, it is easy to check by the continuity properties of Young integrals that $x:=\theta+w$ must be a solution to the YDE associated to $b$, with initial data $x_0$.

\textit{Step 2: Averaging formulation.} Reasoning as in the proof of Proposition~\ref{sectemp prop smooth case}, it can be shown that $\theta$ is also a solution of~\eqref{sectemp nonlinear YDE}.

\textit{Step 3: Separation property.} Given any two solutions $x^1$, $x^2$ for the same initial data $x_0$, $x^i=\theta^i+w$ with $\theta^i\in C^H_t$, we claim that their difference $v=x^1-x^2=\theta^1-\theta^2$ satisfies a linear YDE of the form
\begin{equation}\label{sectemp proof uniqueness linear YDE}
\mathd v_t= v_t\cdot \mathd V_t,\quad V_t=\int_0^t \int_0^1 \nabla \Gamma^w b(\mathd s, \lambda \theta^1_s+(1-\lambda)\theta^2_s)\mathd\lambda.
\end{equation}
This follows from the general fact that for any $\theta^i$ as above and any $A\in C^\gamma_t C^{1+\eta}_{x,loc}$, it holds
\[\int_0^t A(\mathd s, \theta^1_s)-\int_0^t A(\mathd s,\theta^2_s) = \int_0^t (\theta^1_s-\theta^2_s)\cdot \mathd V[A]_s,\quad V[A]_t :=\int_0^1 \int_0^t \nabla A(\mathd s, \lambda \theta^1_s+(1-\lambda)\theta^2_s)d\lambda \]
which can be shown by going through the same proof as in Lemma~6 from~\cite{galeati2020noiseless}.

\textit{Step 4: Conclusion.} The difference $v=x^1-x^2$ satisfies a linear YDE with initial data $v_0=0$. Uniqueness for such equations is well known, thus necessarily $v\equiv 0$.
\end{proof}
Our general aim is to show that the introduction of suitable perturbations $w$ allows to restore existence and uniqueness for the SDE and provides a consistent solution theory even when $b$ is merely distributional;
the next lemmas show that, when it is possible to carry out this program, we can also recover our generalised solutions as limits of those associated to more classical YDEs of the form~\eqref{sectemp YDE} with regular coefficients.

\begin{lem}\label{secyoung lem approximations}
Consider sequences $b^n$ of regular functions (e.g. in $C^2_b$), $x^n_0\in \RR^d$ and $w^n\in C^\delta_t$ with $\delta+H>1$; denote by $x^n$ the unique solution starting from $x_0^n$ to the classical YDE
\begin{equation*}
\mathd x^n = b^n(x^n)\,\mathd \beta +\mathd w^n. 
\end{equation*}
Suppose that
\[
x^n_0\to x_0 \text{ in } \mathbb{R}^d, \quad w^n\to w \text{ in } C^0_t, \quad \Gamma^{w_n} b_n\to A \text{ in } C^{\gamma}_t C^{1+\eta,\lambda}_x
\]
where $\gamma,\eta,\lambda$ are parameters satisfying $\gamma>1/2$, $\gamma(1+\eta)>1$ and $\eta+\lambda\leq 1$. Then $x^n$ converge uniformly to $w+\theta$, where $\theta$ is the unique solution starting from $\theta_0:=x_0-w_0$ to the nonlinear YDE associated to $A$.
\end{lem}

\begin{proof}
We know from Proposition \ref{sectemp prop smooth case} that in the smooth case, $\theta^n:=x^n-w^n$ is a solution to the nonlinear YDE associated to $(\Gamma^{w^n} b^n, x^n_0-w^n_0)$, where the multiplicative averaging operator $\Gamma^{w^n} b^n$ is classically defined pointwise and by hypothesis $(\Gamma^{w^n} b^n, x^n_0-w^n_0)\to (A, \theta_0)$ in $C^\gamma_t C^{1+\eta,\lambda}_x\times \RR^d$. It then follows from Theorem~\ref{secyoung comparison theorem} that $\theta^n\to \theta$ in $C^\gamma_t$; since $w^n\to w$, it follows that $x^n=w^n+\theta^n\to w+\theta$.
\end{proof}

We stated the previous result in a general fashion, so that it can be applied even in situations in which after the limit $w$ does not belong to $C^\delta_t$ with $\delta>1-H$. In this case the analytic definition of $\Gamma^w b$ breaks down, even in the distributional sense, regardless the regularity of $b$; therefore we must invoke the stochastic construction of $\Gamma^w b$ from Section~\ref{sec: avg fields w multiplicative noise}, which truly relies on $\beta$ being sampled as an fBm. However, in the regime $H+\delta>1$, if the regularity of $\Gamma^w b$ is known, the approximating sequence can be constructed explicitly and we obtain the following result, which holds for any given continuous path $\beta\in C^H_t$, not necessarily sampled as a stochastic process.

\begin{prop}\label{sectemp prop general case approximation}
Let $b\in\mathcal{D}(\RR^d)$ be such that $\Gamma^w b \in C^\gamma_t C^{1+\eta,\lambda}_x$ for some $\gamma,\eta,\lambda$ satisfying the usual conditions. Then for any $\theta_0\in\RR^d$ there exists a unique solution $\theta\in C^\gamma_t$ to the nonlinear YDE
\begin{equation}\label{eq:general NYE}
\theta_t = \theta_0 +\int_0^t \Gamma^w b(\mathd s,\theta_s).
\end{equation}
Moreover, denoting by $b^\varepsilon$ a sequence of mollifications of $b$ and by $x^\varepsilon$ the solutions associated to
\[ x^\varepsilon_t = \theta_0 +\int_0^t b^\varepsilon (x^\varepsilon_s)\mathd\beta_s +w_t,\]
then setting $\theta^\varepsilon=x^\varepsilon-w$, it holds $\theta^\varepsilon\to \theta$ in $C^\gamma_t$ as $\varepsilon\to 0$.
\end{prop}
\begin{proof}
The first claim follows from Theorem~\ref{secyoung thm existence and uniquness of YDE}. By Lemma~\ref{lem approximation averaging}, $\Gamma^w b^\varepsilon$ are uniformly bounded in $C^\gamma_t C^{1+\eta,\lambda}_x$ and they are converging to $\Gamma^w b$ in $C^{\gamma'}_t C^{1+\eta',\lambda}_x$ for any $\gamma'<\gamma$ and $\eta'<\eta$; we can choose them so that $\gamma'>1/2$, $ \gamma'(1+\eta')>1$, $\eta'+\lambda\geq 1$. The conclusion then follows from Lemma~\ref{secyoung lem approximations}.
\end{proof}

\subsection{Concepts of existence and uniqueness}\label{sec: concepts solution}

Given parameters $\gamma, \eta, \lambda\in (0,1)$, we will assume throughout this section that they satisfy
\begin{equation}\label{eq: hypothesis parameters}
    \gamma>1/2,\quad \gamma(1+\eta)>1,\quad \eta+\lambda\leq 1 \tag{H}
\end{equation}

\begin{defn}\label{defn: solution}
Let $\{\beta_t\}_{t\in [0,T}$ be a fBm of Hurst parameter $H>1/2$ defined on a probability space $(\Omega,\mathcal{F},\mathbb{P})$, $w$ a continuous deterministic path and $b$ a distributional field. We say that a process $x$ is a \textit{pathwise solution} starting at $x_0\in \mathbb{R}^d$ to the SDE
\[
\mathd x_t = b(x_t)\mathd \beta_t + \mathd w_t
\]
if there exist parameters $\gamma, \eta, \lambda$ satisfying \eqref{eq: hypothesis parameters} and a set $\Omega'\subset\Omega$ of full probability such that, for all $\omega\in \Omega'$, the following hold:
\begin{itemize}
    \item[\textit{i.}] $\Gamma^w b$ is well defined in the sense of Theorem~\ref{thm: main result sec3} and $\Gamma^w b(\omega)\in C^\gamma_t C^{\eta,\lambda}_x$.
    \item[\textit{ii.}] $x(\omega)_0=x_0$ and $x(\omega)\in w+C^\gamma_t$.
    \item[\textit{iii.}] $\theta(\omega):=x(\omega)-w$ satisfies the nonlinear YDE
    \begin{equation*}
        \theta_t(\omega)=\theta_0 + \int_0^t \Gamma^w b(\omega)(\mathd s, \theta_s(\omega)).
    \end{equation*}
\end{itemize}
\end{defn}

Let us comment on the above definition. First of all observe that no filtration on the space $(\Omega,\mathcal{F},\mathcal{P})$ is considered and no adaptability is required on the process $x$. Secondly, the equation satisfied by $\theta_t (\omega)$ is analytically meaningful, once $\Gamma^w b(\omega)$ has the prescribed regularity. In this sense, it is a \textit{random solution} to a \textit{random YDE} rather than a solution to an SDE; in other terms, differently from classical SDEs driven by Brownian motion, all integrals appearing are pathwise defined, which is why we chose the terminology of \textit{pathwise solution}.

Our definition is is some sense closer in spirit to the concept of superposition solution considered in~\cite{flandoli2009remarks} (which is itself inspired by the one from~\cite{ambrosio2004transport}) than to classical concepts of solutions for SDEs. Another way to see it is to define, for $\gamma,\eta,\lambda$ as in Definition~\ref{defn: solution} and for any $A\in C^\gamma_t C^{\eta,\lambda}_x$, $\theta_0\in\mathbb{R}^d$ the set
\begin{equation}\label{eq: solution set}
    C(\theta_0,A):=\bigg\{ \theta\in C^\gamma_t : \theta_t = \theta_0 + \int_0^t A(\mathd s,\theta_s) \  \forall\, t\in [0,T]  \bigg\}.
\end{equation}
Then conditions~\textit{i.} and~\textit{iii.} from Definition~\ref{defn: solution} may be written as
\begin{equation*}
    \mathbb{P}\Big(\omega\in \Omega:\, \Gamma^w b(\omega)\in C^\gamma_t C^{\eta,\lambda}_x,\ \theta(\omega)\in C(\theta_0,\Gamma^w b(\omega))\Big) = 1
\end{equation*}
which can be interpreted as the fact that $\theta$, as a random variable on $C^\gamma_t$, is concentrated on the random set $\omega\mapsto C(\theta_0,\Gamma^w b(\omega))$; we will soon rigorously show that this defines a random set, but let us proceed in the discussion for the moment. As a consequence, if $C(\theta_0,\Gamma^w b(\omega))$ is a singleton for $\mathbb{P}$-a.e. $\omega$, then $\theta$ is uniquely determined. This motivates the following definition.

\begin{defn}\label{defn: path-by-path wellposedness}
Let $\beta$, $w$, $b$ and the parameters $\gamma,\eta,\lambda$ be as in Definition~\ref{defn: solution}. We say that \textit{path-by-path wellposedness} holds for the SDE if
\begin{equation}\label{eq: path-by-path wellposedness}
    \mathbb{P}\Big(\omega\in\Omega: \Gamma^w b(\omega)\in C^\gamma_t C^{\eta,\lambda}_x,\ Card(C(\theta_0,\Gamma^w b(\omega)))=1 \text{ for all } \theta_0\in\mathbb{R}^d\Big)=1.
\end{equation}
\end{defn}

We adopt this terminology, instead of the more classical \textit{path-by-path uniqueness}, to stress the fact that the ``good set'' of full probability on which uniqueness holds is the same for all $\theta_0\in\mathbb{R}^d$, differently from the original result by Davie from~\cite{davie2007uniqueness}. 

\begin{rem}\label{rem: path-by-path law requirement}
By the construction from Theorem~\ref{thm: main result sec3}, the random field $\Gamma^w b$ is adapted to the filtration generated by $\beta$, $\Gamma^w b=\Gamma^w b(\beta)$; therefore~\eqref{eq: path-by-path wellposedness} is exclusively a requirement on the law of $\beta$ and does not depend on the specific probability space $(\Omega,\mathcal{F},\mathbb{P})$ in consideration.
\end{rem}

As a consequence of the theory outlined in Section~\ref{sec:Non-linear Young integration and equations}, we immediately deduce the following.

\begin{lem}\label{lem: conditions path-by-path wellposedness}
Let $\beta$, $w$, $b$ and the parameters $\gamma,\eta,\lambda$ be as in Definition~\ref{defn: solution} and suppose that
\[ \mathbb{P}\big(\omega\in\Omega: \Gamma^w b(\omega)\in C^\gamma_t C^{1+\eta,\lambda}_x\big)=1.
\]
Then path-by-path wellposedness holds for the SDE.
\end{lem}

The rest of the section is dedicated to the proof that $\omega\mapsto C(\theta_0,\Gamma^w b(\omega))$ is a random set, as well as some of its properties. Thus, we believe that it contains results of independent interest regarding nonlinear YDEs.

Before proceeding further, we need to recall a few things on random sets; for a more detailed exposition we refer to~\cite{castaing2006convex}. Given a complete vector space $(E,d)$, the distance between $a\in E$ and a compact $K\subset E$ is given by
\begin{equation*}
    d(a,K)=\inf_{b\in K} d(a,b) = \min_{b\in K} d(a,b)
\end{equation*}
where the infimum is realised since $K$ is compact. Given $K_1$, $K_2$ compact subsets of $E$, their Hausdorff distance $d_H$ is defined as
\begin{equation*}
    d_H(K_1,K_2) = \max\Big\{ \sup_{a\in K_1} d(a,K_2),\, \sup_{b\in K_2} d(b,K_1) \Big\}.
\end{equation*}
Setting $\mathcal{K}(E)=\{K\subset E : K \text{ compact} \}$, $(\mathcal{K}(E),d_H)$ is a complete metric space and moreover we have the identity
\begin{equation*}
    d_H (K_1,K_2) = \sup_{a\in E} |d(a,K_1)-d(a,K_2)| = \max_{a\in K_1\cup K_2} |d(a,K_1)-d(a,K_2)|.
\end{equation*}
Consider $(\mathcal{K}(E),d_H)$ endowed with its Borel $\sigma$-algebra, and let  $(F,\mathcal{A})$ be another measurable space; then it can be shown that a map $X:(F,\mathcal{A})\to (\mathcal{K}(E),d_H)$ is measurable if and only if the map $d(a, X(\cdot))$ is measurable from $(F,\mathcal{A})$ to $(\mathbb{R},\mathcal{B}(\mathbb{R}))$, for all $a\in E$.
Given a probability space $(\Omega,\mathcal{F},\mathbb{P})$, a \textit{random compact set} is a measurable map $X:(\Omega,\mathcal{F},\mathbb{P})\to (\mathcal{K}(E),d_H)$.

\begin{prop}\label{prop: measurable set}
Let $\gamma, \eta, \lambda$ be parameters satisfying~\eqref{eq: hypothesis parameters}. Then for any $\theta_0\in\mathbb{R}^d$ and $A\in C^\gamma_t C^{\eta,\lambda}_x$, the set $C(\theta_0,A)$ is a non-empty, compact subset of $C^\gamma_t$. Moreover the map
\[
(\theta_0,A)\mapsto C(\theta_0,A)
\]
is measurable from $\mathbb{R}^d\times C^\gamma_t C^{\eta,\lambda}_x$ to $\mathcal{K} (C^\gamma_t)$.
\end{prop}

\begin{proof}
The fact that $C(\theta_0,A)$ is non-empty follows from Theorem~3.1 from~\cite{hu2017nonlinear}. By the a priori estimate~\eqref{secyoung a priori estimate}, $C(\theta_0,A)$ is bounded in $C^\gamma_t$; therefore given a sequence $\{\theta^n\} \subset C(\theta_0,A)$, by Ascoli--Arzel\`a we can extract a subsequence (not relabelled for simplicity) such that $\theta^n\to \theta$ in $C^{\gamma-\varepsilon}_t$ for any $\varepsilon>0$. Choosing $\varepsilon$ sufficiently small such that $\gamma+\eta(\gamma-\varepsilon)>1$, it follows from the continuity of Young integrals that
\[
\theta^n_\cdot = \theta_0+\int_0^\cdot A(\mathd s,\theta^n_s) \to \theta_0 + \int_0^\cdot A(\mathd s,\theta_s)=\theta_\cdot \quad \text{in } C^\gamma_t.
\]
Namely, $\theta^n$ converge in $C^\gamma_t$ to an element of $C(\theta_0,A)$, which shows compactness.\\
In order to prove the second claim, it is enough to show that for any $y\in C^\gamma_t$, the map
\[
\mathbb{R}^d\times C^\gamma_t C^{\eta,\lambda}_x \ni (\theta_0,A)\mapsto d(y, C(\theta_0,A))\in \mathbb{R}
\]
is measurable; we will actually show that it is lower semicontinuous.
Fix $y\in C^\gamma_t$ and let $(\theta^n_0,A^n)\to (\theta_0,A)$; by compactness of $C(\theta^n_0,A^n)$, for each $n$ there exists $\theta^n\in C(\theta^n,A^n)$ such that $d(y,\theta^n)=d(y,C(\theta^n,A^n))$.
Up to extracting a subsequence which realizes the liminf, we can assume without loss of generality that $\lim d(y,C(\theta^n,A^n))$ exists; as the sequence $(\theta^n_0,A^n)$ is convergent, it must also be bounded, which implies by~\eqref{secyoung a priori estimate} that $\{\theta^n\}_n$ is bounded in $C^\gamma_t$.
Invoking Ascoli--Arzel\`a and reasoning as in the previous point, using the continuity of nonlinear Young integrals, we can find a (not relabelled) subsequence such that $\theta^n\to\theta\in C(\theta_0,A)$ in $C^\gamma_t$. As a consequence
\[
d(y,C(\theta_0,A)) \leq d(y,\theta) = \lim_{n\to\infty} d(y,\theta_n) = \liminf_{n\to\infty} d(y,C(\theta^n_0,A^n))
\]
which implies lower semicontinuity, and thus concludes the proof.
\end{proof}

The fact that $C(\theta_0,\Gamma^w b(\omega))$ is a random set follows from the following more general result.

\begin{cor}
Let $(\Omega,\mathcal{F},\mathbb{P})$ be a probability space on which  a random field $A=A(\omega)\in C^\gamma_t C^{\eta,\lambda}_x$ and a random vector $\xi=\xi(\omega)\in\mathbb{R}^d$ is defined. Then the map
\begin{equation*}
    \omega\mapsto C(\xi(\omega), A(\omega))
\end{equation*}
defines a random compact subset of $\mathcal{C}^\gamma_t$.
\end{cor}

\begin{proof}
It is an immediate consequence of Proposition~\ref{prop: measurable set} and the fact that composition of measurable functions is measurable.
\end{proof}

\subsection{Proofs of the main results}\label{sec: proofs main results}

The goal is to find specific conditions on the parameters $H$, $\delta$ and the regularity of $b$ in order to obtain existence and uniqueness of  \eqref{sectemp YDE}.  To this end, we will distinguish our analysis into to different cases: when $b\in C^\alpha_x$ with $\alpha\in (0,1)$, we will find conditions for $\delta,H$ and $\alpha$ through application of Proposition \ref{sectemp prop uniqueness mildly regular case} to ensure existence of \eqref{sectemp YDE}. For the general case of $b\in \cD(\RR^d)$, we will consult Proposition \ref{sectemp prop general case approximation} to find conditions for $\delta,H,$ and $\alpha$ such that existence and uniqueness holds.\\

We are now ready to give the proofs of our main results.

\begin{proof}[Proof of Theorem~\ref{main thm1}]
It follows from Corollary \ref{cor: regualrity of multiplcative field} and Remark \ref{rem: sufficent regularity requirement} that, under the regularity assumption $T^w b\in C^\gamma_t C^2_x$, the multiplicative averaged field $\Gamma^w b$ is a well defined random field and we can find $\gamma', \eta,\lambda\in (0,1)$ such that
\[
\Gamma^w b \in C^{\gamma'}_t C^{1+\eta,\lambda}_x \qquad\mathbb{P}\text{-a.s.}
\]
together with $\gamma'>1/2$, $\gamma'(1+\eta)>1$ and $\eta+\lambda\leq 1$. Therefore path-by-path wellposedness follows from Lemma~\ref{lem: conditions path-by-path wellposedness}. Given two pathwise solutions $x^i=\theta^i+w$ starting at $x_0$, setting $\theta_0=x_0-w_0$, it holds
\begin{align*}
    \PP \big(x^1=x^2 \text{ in } C^0_t \big)
    & = \PP \big( \theta^1=\theta^2 \text{ in } C^0_t \big)
    \geq \PP \big( \theta^i\in C^{\gamma'}_t, \Gamma^w b\in C^{\gamma'}_t C^{\eta,\lambda}_x, \theta^i\in C(\theta_0,\Gamma^w b) \big)\\
    & \geq \PP(\Gamma^w b\in C^{\gamma'}_t C^{\eta,\lambda}_x, C(\theta_0,\Gamma^w b)\text{ is a singleton}\big)=1
\end{align*}
which shows indistinguishability.
Adaptedness follows from the formula $\theta(\omega)_\cdot=\mathcal{I}(\Gamma^w b(\omega))(\cdot,\theta_0)$ and the fact that by construction the field $\Gamma^w b$ is adapted to $\beta$, in the sense that $\{\Gamma^w_s b,s\in [0,t]\}\subset \sigma\{\beta_s:s\in [0,t]\}$.
Finally, formula \eqref{eq: form solution} follows from the one for $\theta$ and the change of variables $x=\theta+w$.
\end{proof}

\begin{proof}[Proof of Proposition~\ref{prop: justification solutions}]
Part~\textit{i.} is just a consequence of Proposition~\ref{sectemp prop smooth case}; in particular it is enough to require $b\in C^2_b$, $w\in C^\delta_t$ with $\delta+H>1$.\\
Under condition~\eqref{main thm1 condition}, by Remark~\ref{rem: approximation mult avg} we can find a sequence $(b^n,w^n)$ (for instance in $C^2_b\times C^\delta_t$) such that $b^n\to b$ in the sense of distributions, $w^n\to w$ uniformly and $\Gamma^{w^n} b^n(\omega)\to \Gamma^w b(\omega)$ in $C^{\gamma'}_t C^{1+\eta,\lambda}_x$ for $\mathbb{P}$-a.e. $\omega$;
moreover we can choose the parameters so that $\gamma'>1/2$, $\gamma'(1+\eta)>1$ and $\eta+\lambda\leq 1$.
Therefore point~\textit{ii.} follows from an application of Lemma~\ref{secyoung lem approximations}.\\
Suppose now $(b^n,w^n)$ is a sequence in $C^2_b\times C^\delta_t$ satisfying the assumptions of point~\textit{iii.};
by properties of classical averaged fields, $T^{w^n} b^n\to T^w b$ in the sense of distributions, which implies that $T^w b\in C^\gamma_t C^2_x$ and $T^{w^n} b^n\to T^w b$ in $C^\gamma_t C^2_x$.
But then by Theorem~\ref{thm: main result sec3} and Remark~\ref{rem: sufficent regularity requirement}, we can find $\gamma',\eta,\lambda$ as above such that $\Gamma^{w^n} b^n\to \Gamma^w b$ in $L^p(\Omega; C^{\gamma'}_t C^{\eta,\lambda}_x)$. The conclusion then follows again from an application of Lemma~\ref{secyoung lem approximations}.
\end{proof}

In order to specialize the above criterion to cases of practical interest, we need the following lemma.

\begin{lem}\label{lem: interpolation averaging}
Let $b\in C^\alpha_x$ for some $\alpha\in \mathbb{R}$, $w$ a continuous path s.t. $T^w b\in C^{1/2}_t C^{\alpha+\nu}_x$. Then
\[
T^w b\in C^\gamma_t C_x^{\alpha+2\nu(1-\gamma)} \qquad \forall\, \gamma\in [1/2,1].
\]
\end{lem}
\begin{proof}
Since $b\in C^\alpha_x$, $T^w b\in C^1_t C^\alpha_x$; the claim then follows from interpolation estimates. Indeed, by Besov interpolation inequality (see \cite[Thm. 2.80]{BahCheDan}), for any $\theta\in [0,1]$ it holds
\[
\Vert T^w_{s,t} b\Vert_{\alpha+(1-\theta)\nu}
\lesssim \Vert T^w_{s,t} b\Vert_{\alpha}^{\theta}\, \Vert T^w_{s,t}b\Vert_{\alpha+\nu}^{1-\theta}
\lesssim |t-s|^{\theta+(1-\theta)/2} \,\Vert T^w_{s,t} b\Vert_{1,\alpha}^{\theta}\,\Vert T^w_{s,t} b\Vert_{1/2,\alpha+\nu}^{1-\theta}
\]
and the conclusion follows by taking $\gamma=(1+\theta)/2$.
\end{proof}

\begin{proof}[Proof of Theorem~\ref{main thm2}]
To show the first statement, we need to verify that under condition~\ref{main thm2 condition}, $T^w b\in C^\gamma_t C^2_x$ for some $\gamma>3/2-H$; by the assumption $T^w b\in C^{1/2}_t C^{\alpha+\nu}_x$ and Lemma~\ref{lem: interpolation averaging}, it is enough to verify that
\begin{equation*}
    \begin{cases}
    \gamma>3/2-H\\
    \alpha+2\nu (1-\gamma)>2
    \end{cases}.
\end{equation*}
It is easy to check that one can find $\gamma\in (0,1)$ satisfying the above conditions if and only if~\eqref{main thm2 condition} holds.
Similar computations show that, under~\eqref{main thm2 condition2}, $T^w b\in C^\gamma_t C^{n+1}_x$, which implies that we can find $\gamma',\eta,\lambda$ satisfying the usual conditions such that $\Gamma^w b\in C^{\gamma'}_t C^{n+\eta,\lambda}_x$; the regularity of the flow then follows from the last part of Theorem~\ref{secyoung thm existence and uniquness of YDE}.
\end{proof}

\begin{proof}[Proof of Theorem~\ref{main thm4}]
The proof follows the same lines of the previous ones, only this time we want to check that the conditions of Proposition~\ref{sectemp prop uniqueness mildly regular case} are met. By the assumptions and Lemma~\ref{lem: interpolation averaging}, $T^w b\in C^\gamma_t C^{\alpha+2\nu(1-\gamma)}_x$ for any $\gamma>1/2$; taking $\gamma>3/2-H$ and applying Corollary~\ref{cor: regualrity of multiplcative field}, we deduce that $\mathbb{P}$-a.s. $\Gamma^w b\in C^{\gamma'}_t C^{1+\eta,\lambda}_x$ for any $\gamma'<\gamma+H-1$, $1+\eta<\alpha + 2\nu (1-\gamma)$ and $\lambda$ sufficiently small. In order to find $\gamma',\eta$ such that $\gamma'+H\eta>1$ it is therefore enough to verify that there exists $\gamma>1/2$ such that
\begin{equation*}
    \begin{cases}
    \gamma>3/2-H\\
    \frac{1}{2} + H (\alpha+2\nu (1-\gamma)-1)>1
    \end{cases}
\end{equation*}
or equivalently
\begin{equation*}
    \begin{cases}
    \gamma>3/2-H\\
    \alpha+2\nu (1-\gamma)>1+\frac{1}{2H}
    \end{cases}
\end{equation*}
Taking $\gamma$ of the form $\gamma=3/2-H+\varepsilon$ with $\varepsilon>0$ sufficiently small, it is easy to check that the above conditions are satisfied under assumption~\eqref{main thm4 condition}.
\end{proof}

\section{Further extensions}\label{sec: further extensions}

\subsection{Time inhomogeneous diffusion coefficient}\label{sec: inhomogeneous coefficients}

So far we assumed the diffusion coefficient $b$ to be homogeneous, in the sense that $b(t,x)=b(x)$. However, our method can be easily extended to the general case of time in-homogeneous $b$. We will outline here the necessary conditions in order to obtain wellposedness of equations with time homogeneous coefficients   of the form
\begin{equation*}
    \dd x_t = b(t,x_t)\dd \beta_t +\dd w_t. 
\end{equation*}

The first step in this direction is to define the multiplicative averaged field $\Gamma^w b$. To this end, it is readily seen that if $(t,x)\mapsto b(t,x)$ is smooth in both variables and $w\in C^\delta_t$ with not too small $\delta$, the analytical definition of $\Gamma^w b$ from Lemma \ref{lem defn distributional averaging} still holds. In fact, if $b\in C_t^\rho C^{\alpha+\eta}_x$ with $\rho>1-H$, $\alpha\in \RR$ and $\eta\in (0,1]$, under the assumption $H+\eta\delta>1$, there exists a unique distribution $\Gamma^wb\in C^H_tC^\alpha_x$ such that 
\begin{equation}\label{eq:young inhomogenous}
    \|\Gamma_{s,t}^wb-b(s,\cdot+w_s)\beta_{s,t}\|_{C^\alpha_x} \lesssim |t-s|^{H+\eta\delta}. 
\end{equation}
Indeed, setting $\Xi_{s,t}=\tau^{w_s}b(s,\cdot)\beta_{s,t}$, we observe that 
\begin{equation*}
    \|\delta \Xi_{s,u,t}\|_{C^\alpha_x} \lesssim \left[\|b(s,\cdot+w_u)-b(u,\cdot+w_u)\|_{C_x^\alpha}+\|b(s,\cdot+w_u)-b(s,\cdot+w_s)\|_{C^\alpha_x}\right]|\beta_{u,t}|. 
\end{equation*}
Invoking the assumptions of H\"older regularity in $t\mapsto b(t,\cdot)$, $w$, and $\beta$, we obtain
\begin{equation*}
    \|\delta \Xi_{s,u,t}\|_{C^\alpha_x} \lesssim \|b\|_{C^\rho_t C^{\alpha+\eta}_x} \llbracket \beta\rrbracket_{C^H_t} (1+\llbracket w\rrbracket_{C^\delta_t})
    |t-s|^{H+\eta\delta\wedge \rho},
\end{equation*}
where we have employed estimates similar to those of Lemma \ref{lem defn distributional averaging}. An application of the sewing lemma then implies \eqref{eq:young inhomogenous}. Thus, from an analytical perspective it is readily seen that the multiplicative averaged field is well defined. In order to obtain the regularizing effect from $w$, we then need to use the stochastic construction of $\Gamma^wb$ by application of Proposition \ref{prop: lemma 3.4 HairerLi}. Lemma \ref{lem: Lp bounds for mult. avg. Op.} is thus readily extended to the time in-homogeneous case, under the assumption that the classical averaged field $T^w b\in C^\gamma_tC^{1+\eta}_x$. For example, in \cite{galeati2020noiseless} it is shown that $T^w b\in C^\gamma_tC^{1+\eta}_x$ for $b\in L^q([0,T];C^\alpha_x)$ with $q>2$ and $\alpha\in \RR$ under suitable conditions on $w$.  For a more detailed analytical construction of the classical averaged field with time in-homogeneous $b$,  see \cite{galeati2020noiseless}. In a similar spirit, one can then readily apply the modified GRR lemma \ref{lem GRR} in order to obtain almost sure space-time H\"older regularity of $\Gamma^w b$. 

With the time in-homogeneous multiplicative averaged field at hand, one can then go through the same abstract procedure for existence and uniqueness of non-linear young equations as shown in section \ref{sec:Non-linear Young integration and equations} by setting $A_{s,t}(x)=\Gamma^w_{s,t}b_s(x)$ in Theorem \ref{secyoung thm existence young integral} and Theorem \ref{secyoung thm existence and uniquness of YDE}. These theorems can then be used to extend the results in section \ref{sec:young eq with multiplicative noise} to allow for time in-homogeneous diffusion coefficients $b$ with possibly distributional spatial dependence.

\subsection{Including a non-Lipschitz drift term}
So far, we have only considered \eqref{intro general SDE} in the case when $b_1\equiv 0$ and $b_2=b$. However, our results immediately extend to equations with both non trivial drift and diffusion, of the form 
\begin{equation*}
    x_t = x_0+\int_0^t b_1(x_s)\dd s + \int_0^t b_2(x_s)\dd \beta_s + w_t,\qquad x_0\in \RR^d.  
\end{equation*}
Again, by the change of variables $\theta=x-w$, we see that $\theta$ formally solves the equation 
\begin{equation*}
    \theta_t=x_0+\int_0^tb_1(\theta_s+w_s)\dd s+\int_0^tb_2(\theta_s+w_s)\dd \beta_s. 
\end{equation*}
Setting 
\begin{equation*}
A_{s,t}(x):=T^w_{s,t}b_1(x)+\Gamma^w_{s,t}b_2(x),
\end{equation*}
we can interpret the equation in the Young integral sense as 
\begin{equation*}
    \theta_t =x_0+\int_0^t A(\dd s,\theta_s)
\end{equation*}
Under the condition that $A$ is sufficiently regular, existence and uniqueness for the YDE holds by Theorem \ref{secyoung thm existence and uniquness of YDE}. It is therefore enough to require $T^wb_1$ and $\Gamma^w b_2$ to belong to $ C_t^\gamma C^{1+\beta,\lambda}_x$ for suitable $\gamma,\beta,\lambda$. Then the results in  Section \ref{sec:young eq with multiplicative noise} can be extended directly. 

One can in this case also invoke time in-homogeneous drift and diffusion $b_1$ and $b_2$ by following the steps outlined in the previous subsection.  

\subsection{Random initial condition}\label{sec: random initial}
So far we have only considered deterministic initial data $x_0\in \mathbb{R}^d$ (resp. $\theta_0=x_0-w_0\in\mathbb{R}^d$).
However, especially in view of applications to optimal transport and fluid dynamics equations, it is often interesting to allow random initial data for the SDE. This extension can be easily implemented in the framework of Section~\ref{sec: concepts solution}, as we are now going to show.
\begin{defn}\label{defn: solution random initial}
Let $(\Omega,\mathcal{F},\mathbb{P})$ be a probability space on which an fBm $\{\beta_t\}_{t\in [0,T}$ of Hurst parameter $H>1/2$, as well as an independent $\mathbb{R}^d$-valued random variable $\xi$, are defined; consider also a continuous deterministic path $w$ and a distributional field $b$. We say that a process $x$ is a \textit{pathwise solution} to the SDE
\[
\mathd x_t = b(x_t)\mathd \beta_t + \mathd w_t,\quad x_0=\xi
\]
if there exist parameters $\gamma, \eta, \lambda$ satisfying \eqref{eq: hypothesis parameters} such that $\Gamma^w b$ is well defined in the sense of Theorem~\ref{thm: main result sec3} and, setting $\theta=x-w$, $\zeta=\xi-w_0$, it holds
\[
\mathbb{P}\Big(\omega\in \Omega \, :\, \Gamma^w b(\omega)\in C^\gamma_t C^{\eta,\lambda}_x,
\ \ \theta(\omega)= C^\gamma_t,
\ \ \theta(\omega)\in C(\zeta(\omega), \Gamma^w b(\omega))\,
\Big) = 1.
\]
\end{defn}

As a consequence of the theory from Section~\ref{sec:Non-linear Young integration and equations}, we deduce the following result.
\begin{cor}
Let $\beta, b, w, \xi,\zeta$ be as above and such that the  assumptions of Lemma~\ref{lem: conditions path-by-path wellposedness} are satisfied. Then any pathwise solution $x$ to the SDE with initial condition $\xi$, $x=\theta+w$, satisfies
\[
\mathbb{P}\Big( \omega\in \Omega\, :\,  \theta(\omega)_t=\mathcal{I}(\Gamma^w b(\omega))(t,\zeta(\omega))\ \text{ for all } t\in [0,T]\Big) = 1
\]
where $\mathcal{I}$ is the map defined in Corollary~\ref{cor: continuous dependence flow map}, i.e. $\mathcal{I}(\Gamma^w b(\omega))$ is the flow associated to $\Gamma^w b(\omega)$. In particular all the conclusions follow if the assumptions of Theorem~\ref{main thm1} are satisfied.
\end{cor}

\section{Concluding remarks}\label{sec:concluding}

We have shown that through a suitable perturbation of a continuous but irregular path $w$, the SDE 
\begin{equation}\label{eq: last SDE}
    \dd x_t = b(x_t)\dd \beta_t +\dd w_t,\qquad x_0\in \RR^d
\end{equation}
is well posed and admits a unique solution even for distributional coefficients $b$ in terms of Definition \ref{defn: solution} and \ref{defn: path-by-path wellposedness}, in the case when $\{\beta_t\}_{t\in [0,T]}$ is a fBm with $H\in (\frac{1}{2},1)$.  
This can be seen as a first step in a more general program of proving regularization of multiplicative SDEs through perturbation by irregular/rough paths. The first question one could ask is whether it is possible to less restrictive requirements on $b$ given a certain regularizing path $w$. For example, 
in \cite{Catellier2016}, \cite{galeati2020noiseless} (and partially related \cite{butkovsky2019approximation}), sharper results are obtained for SDEs with additive drift (non multiplicative case) by exploiting Girsanov transform. If $w$ is sampled as an fBm of parameter $\delta$, another possible way to solve the SDE in \eqref{eq: last SDE} (say for $x_0=0$ wlog) would be to check that the process
    \begin{equation*}
        \tilde{w}_t = w_t -\int_0^t b(w_s)\dd \beta_s     \end{equation*}
    is again an fBm of parameter $\delta$ under a new probability law $\mathbb{Q}$; if that's the case, then $w$ itself is a solution to the equation w.r.t. $\tilde{w}$. However, the estimates from Proposition \ref{prop: lemma 3.4 HairerLi} are not enough to establish exponential integrability and thus to check if Novikov holds.
Another possibility to obtain sharper results could be to  apply the recently developed stochastic sewing lemma \cite{le2020stochastic}, in combination with a more direct application of the results obtained by Hairer and Li in \cite{hairer2019averaging}.  Probably in that case, existence and uniqueness in the class of adapted processes is more straightforward. Our results on the other hand have the advantages that: i) uniqueness also holds without adaptability requirements (although a posteriori the unique solution will be adapted); ii) existence and uniqueness of solutions immediately comes with a regular flow (which is quite difficult to establish by means of stochastic techniques); iii) the resulting equation has a pathwise analytical meaning, its randomicity being in the random field $\Gamma^w b$ but not the YDE itself.

A possibly more challenging extension of our results, is to consider the case of multiplicative fBm with $0<H\leq \frac{1}{2}$. As seen through our analysis, such an extension would be highly dependent on showing the relation between the multiplicative averaged field $\Gamma^wb$ with the classical averaged field $T^wb$ when $\Gamma^w$ is driven by a fBm with $H\leq \frac{1}{2}$. In this case,  Proposition \ref{prop: lemma 3.4 HairerLi} breaks down, and  thus a similar statement in the rough case would be needed.  Furthermore, if one can prove that $\Gamma^wb\in C_t^\gamma C^\eta_{x,loc}$ for general distributions $b$, one can not hope for a $\gamma>\frac{1}{2}$, which is required to apply the non-linear Young formalism employed in this article. To this end, one could hope to use techniques developed on nonlinear rough paths (see e.g. \cite{nualart2019nonlinear,coghi2019rough}), but the exact formulation of the equation in this context is not completely clear. 

Observe that for smooth functions $b$ and under the assumption that $H+\delta>1$ (recall that $\delta\in (0,1)$ is the H\"older regularity of $w$) it holds that 
\begin{equation*}
    \Gamma^wb=\Gamma^w(b\ast \delta_0)=b\ast \Gamma^w \delta_0= b\ast \bar{\nu}^w,
\end{equation*}
where $\bar{\nu}^w$ is the reflection of $\nu^w$ formally given by
\begin{equation*}
    \nu^w_{s,t}=\int_s^t \delta_{w_r}\dd \beta_r, 
\end{equation*}
and for $y\in \RR^d$, $\delta_y$ denotes the Dirac delta centered at $y$. It is tempting to think of  $\nu^w$ as being a form of "weighted occupation measure". However, in general $\nu^w$ will NOT be a measure. Anyway, applying the approximation procedure from Section~\ref{sec: avg fields w multiplicative noise}, the above relation is preserved also in the case $H+\delta\leq 1$, once interpreted as random variables: for fixed $b$,
    \[
    \Gamma^w b(\omega) = b\ast \tilde{\nu^w}(\omega) \quad\text{ for } \mathbb{P}\text{-a.e. }\omega\in\Omega.
    \]
    Now on the r.h.s. the random variable appearing does not depend on $b$ anymore, so it can be regarded as a regular version of the family of random variables $\{\Gamma^w b\}_{b\in E}$: once we fix the set $\Omega'\subset\Omega$ on which $\nu^w$ is defined and regular, so are $\Gamma^w b$. In this sense, in many considerations we could also make the full probability set independent of $b$, deriving the regularity of $\Gamma^w b$ from that of $\nu^w$ and Young's convolution inequality, which can then be seen analogously to constructing the classical averaged field as a convolution between a function $b$ and the reflected local time associated to $w$. 

One could also readapt the concept of $\rho$-irregularity (see e.g. \cite{galeati2020prevalence}) in this setting. Indeed at least formally, convolution with $\nu^w$ coincides at the Fourier level to a Fourier multiplier of the form
    \[
    \hat{\nu}^w(\xi) = \int_s^t e^{i\xi\cdot w_r} \mathd\beta_r
    \]
    where for any fixed $\xi$, $\hat{\nu}^w(\xi)$ is a well defined random variable (random path actually, once we apply Kolmogorov) by the Lemma from \cite{hairer2019averaging}. Combining this with the classical $\rho$-irregularity property, one should obtain that if $w$ is $(\gamma,\rho)$-irregular, then for any $\gamma'<\gamma+H-1$, $\rho'<\rho$ it holds
    \[
    \EE \big[\,\Vert \hat\nu^w(\xi)\Vert_{\gamma'}^p\big]^{1/p}\lesssim |\xi|^{-\rho'}
    \]
    One could then ask the more difficult question of whether it's possible to establish that
    \[
    \PP \left( \sup_{\xi\in\RR^d} |\xi|^{\rho'} \Vert \hat\nu^w(\xi)\Vert_{\gamma'} <\infty \right) =1  
    \]
    which would be a true analogue of the $\rho$-irregularity property.

\bibliographystyle{plain}

\end{document}